\documentclass[11pt]{amsart}
\usepackage{amssymb, amsmath}
\usepackage{amsthm}
\usepackage{amscd}
\usepackage{graphicx}
\newtheorem{theorem}{Theorem}
\newtheorem{lemma}[theorem]{Lemma}

\newtheorem{corollary}[theorem]{Corollary}

\theoremstyle{definition}

\newtheorem*{acknowledgement}{Acknowledgement}

\title{Gluing locally symmetric manifolds: asphericity and rigidity}
\author{T. T$\hat{\mathrm{a}}$m Nguy$\tilde{\hat{\mathrm{e}}}$n Phan}
\address{Department of Mathematics\\
5734 S. University Ave.\\
Chicago, IL 60637}
\email{ttamnp@math.uchicago.edu}
\DeclareMathOperator{\Flag}{Flag}
\DeclareMathOperator{\Ad}{Ad}

\DeclareMathOperator{\Ker}{Ker}
\DeclareMathOperator{\Lk}{Lk}

\DeclareMathOperator{\Isom}{Isom}
\DeclareMathOperator{\Out}{Out}

\DeclareMathOperator{\SL}{SL}
\DeclareMathOperator{\SO}{SO}
\DeclareMathOperator{\CAT}{CAT}

\DeclareMathOperator{\Cone}{Cone}

\DeclareMathOperator{\Fix}{Fix}

\DeclareMathOperator{\Homeo}{Homeo}
\DeclareMathOperator{\Conj}{Conj}
\DeclareMathOperator{\St}{St}
\DeclareMathOperator{\st}{st}

\def\R{\mathbb{R}}
\def\Z{\mathbb{Z}}
\def\N{\mathbb{N}}
\def\Q{\mathbb{Q}}

\def\X{\overline{X}}
\def\M{\overline{M}}

\input xy
\xyoption{all}
\begin{document}
\begin{abstract}
We use the reflection group trick to glue manifolds with corners that are Borel-Serre compactifications of locally symmetric spaces of noncompact type and obtain aspherical manifolds. We call these \emph{piecewise locally symmetric} manifolds. This class of spaces provide new examples of aspherical manifolds whose fundamental groups have the structure of a complex of groups. These manifolds typically do not admit a locally $\CAT(0)$ metric. We prove that any self homotopy equivalence of such manifolds is homotopic to a homeomorphism. We compute the group of self homotopy equivalences of such a manifold and show that it can contain a normal free abelian subgroup, and thus can be infinite.
\end{abstract}
\maketitle
\section{Introduction}
There has been a desire to construct more examples of aspherical manifolds, to enrich the list of known examples or to disprove various conjectures on aspherical manifolds. Gluing constructions have been playing a crucial role in giving interesting examples of such manifolds. For example, the Davis construction using the reflection group trick gives examples of closed aspherical manifolds whose universal covers are not homeomorphic to any Euclidean spaces.
\newline

\textbf{Piecewise locally symmetric manifolds.} In this paper we introduce a class of aspherical manifolds, which we call \emph{piecewise locally symmetric manifolds}, that are obtained by applying the reflection group trick to manifolds with corners that are the Borel-Serre compactification of complete, finite volume, irreducible, locally symmetric spaces of noncompact type. We will recall the reflection group trick and define these manifolds in Section~\ref{sec:definitions}. (Un)fortunately, these manifolds satisfy the Hopf conjecture. However, they have interesting properties: homotopy equivalences of such manifolds are homotopic to a homeomorphism; they can have infinite groups of self homotopy equivalences; and they form a new class of examples of spaces whose fundamental groups have the structure of a complex of groups.

We begin with a simple example. Let $X$ be the product of two surfaces $S_1\times S_2$, each of which has one boundary component $\partial S_i$. Then topologically, $X$ is a manifold with boundary but smoothly, a manifold with corners. The boundary $\partial X$ is the union of its two codimension $1$ corners $S_1\times\partial S_2$ and $\partial S_1\times S_2$ along its codimension $2$ corner $\partial S_1\times\partial S_2$. If we double $X$ along its boundary, we will get a manifold $M$ with $\pi_2(M)$ nontrivial since there are noncontractible loops in the boundary that are contractible in the manifold. To get an aspherical manifold, we can proceed as follows. 

Let $D_3$ be the dihedral group of order $6$. Then $D_3$ acts on the Euclidean plane $\R^2$ as the group generated by reflections across two lines at angle $\pi/3$ with each other at the origin. A fundamental domain of $D_3$ is a sector $D$ with angle $\pi/3$, which is a manifold with corners with two codimension $1$ corners and one codimension $2$ corner.  One can recover the plane (which is a manifold) by gluing $6$ copies of $D$ (which is a manifold with corners) in a natural way. 

In the same way, we can glue six copies of the above $X$ to get a closed manifold $M$ on which $D_3$ acts with a fundamental domain a copy of $X$. The fundamental group of $M$ has the structure of a hexagon of groups.

Some more interesting examples are the following. Given an arithmetic manifold, e.g. $\SO(3,\R)\backslash\SL(3,\R)/\Gamma$, for some torsion-free subgroup $\Gamma$ of $\SL(3,\Z)$, its Borel-Serre compactification is a manifold $Y$ with corners. One can do the same thing as above with $Y$ (instead of $X$) and an appropriate Coxeter group $W$ (instead of $D_3$) to obtain an manifold $M$ on which $W$ acts with a fundamental domain a copy of $Y$. The manifold $M$ is compact if and only if $W$ is finite. If $W$ is not finite, one can get a compact manifold by taking the quotient of $M$ by a torsion-free, finite index subgroup of $W$. This is the general idea of the reflection group trick. 

The above manifolds are examples of piecewise locally symmetric manifolds. The first main result of this paper is that they are aspherical.

\begin{theorem}[Asphericity]\label{asphericity}
Piecewise locally symmetric manifolds are aspherical.
\end{theorem}

One of the main steps in proving asphericity of piecewise locally symmetric manifolds is to prove that their fundamental groups are nonpositively curved complexes of groups.  There have not been so many examples of nonpositively curved complexes of groups of dimension $\geq 2$. Piecewise locally symmetric manifolds provide many new examples of nonpositively curved complexes of groups.
\newline

\textbf{(Non)rigidity.}
Since piecewise locally symmetric manifolds are made of pieces that are strongly rigid by Mostow-Prasad-Margulis rigidity, one natural question to ask is if or to what extent rigidity holds for this class of manifolds. In particular, we address the following question.

For $M$ be a manifold. We denote by $\Homeo(M)$ the group of self-homeomorphisms of $M$ and by $\Homeo_0(M)$ the group of self-homeomorphisms of $M$ that are homotopic (not necessarily isotopic) to the identity map. Hence, $\Homeo(M)/\Homeo_0(M)$ is the group of self-homeomorphisms of $M$ up to homotopy. The action of a homeomorphism on $\pi_1(M)$ induces a natural homomorphism 

\begin{equation}\label{*}
\eta \colon \Homeo(M)/\Homeo_0(M) \longrightarrow \Out(\pi_1(M)).
\end{equation}
If $M$ is aspherical, then $\eta$ is always an injection since $\Out(\pi_1(M))$ is canonically isomorphic to the group of self homotopy equivalences of $M$ up to homotopy. One can ask if $\eta$ is a surjection, i.e. if a homotopy equivalence of $M$ is homotopic to a homeomorphism.

It is known for a number of classes of manifolds that $\eta$ is an isomorphism. These include closed surfaces, by the Dehn-Nielsen-Baer Theorem; infra-nilmanifolds, by Auslander \cite{Auslander};  finite-volume, complete, irreducible, locally symmetric, nonpositively curved manifold of dimension greater than $2$, by Mostow Rigidity; closed, nonpositively curved manifolds of dimension greater than $4$, by Farrell and Jones \cite{FJ2}, and for solvmanifolds, by work of Mostow \cite{Mostow}.

\begin{theorem}[Rigidity/Nielsen realization]\label{Mostow}
Let $M$ be a piecewise locally symmetric manifold of dimension $n > 2$. Assume that the pieces of $M$ are irreducible. Then any homotopy equivalence of $M$ is homotopic to a homeomorphism. Moreover, $\Out(\pi_1(M))$ can be realized as a group of homeomorphisms. That is, there is an injective homomorphism $\rho \colon \Out(\pi_1(M)) \longrightarrow \Homeo(M)$ such that 
\[ \xymatrix{
& & \Homeo(M) \ar[d]^p  \\
 \Out(\pi_1(M)) \ar[urr]^\rho &  & \Homeo(M)/\Homeo_0(M)\ar[ll]_\eta  }. \]
\end{theorem}

That $\Out(\pi_1(M))$ can be realized as group of homeomorphisms of $M$ means that there is a global solution to the Nielsen Realization problem. In general, one does not expect such a global solution. Morita \cite{Morita} proved that if $M$ is a surface of genus greater than $17$, the group $\Out(\pi_1(M))$ does not lift to $\Homeo(M)$.   

In proving Theorem \ref{Mostow}, we obtain the following rigidity property of the decomposition of $M$ into locally symmetric pieces.

\begin{theorem}\label{piece rigidity}
Let $M$ be a piecewise locally symmetric manifold of dimension $n > 2$, and let $M_i$, $i\in I$, be the locally symmetric pieces in the decomposition of $M$. Suppose that $M_i$ are irreducible and have $\Q$-rank $> 1$. Let $f \colon M_i \longrightarrow M$ be a $\pi_1$-injective map. Then $f$ is homotopic to a map $g \colon M_i \longrightarrow M$ that is a diffeomorphism onto a piece $M_j \subset M$. 
\end{theorem}
We believe that the above theorem is true for the case where the pieces have $\Q$-rank $1$, but the technique we use to prove the Theorem \ref{piece rigidity} does not apply to the $\Q$-rank $1$ case. The key in the proof of this theorem is to use nonpositively curved complexes of groups and apply a theorem of Farb in \cite{Farb} on generation of lattices with $\Q$-rank $\geq 2$ by certain nilpotent subgroups. In some way, the mechanism for rigidity is the same as for asphericity, i.e. both are governed by the nonpositively curved complex (of smaller dimension than $M$) given by the complex of groups structure of $\pi_1(M)$.  

Theorem \ref{piece rigidity} means that the embedding of each piece $M_i$ in $M$ is rather rigid up to homotopy.  It follows that the decomposition of $M$ into locally symmetric pieces is unique. In other words, up to homotopy there is only one way to tile $M$ by locally symmetric pieces, and that any homotopy equivalence of $M$ must preserve the decomposition. This is actually the first step in proving Theorem \ref{Mostow}. Pushing to homeomorphisms is the other step.

Despite the rigidity that Theorem \ref{piece rigidity} seems to suggest, there is a nonrigidity aspect of piecewise locally symmetric manifolds that appears in the structure of $\Out(\pi_1(M))$. The group $\Out(\pi_1(M))$ contains a (possibly trivial) free abelian normal subgroup $\mathcal{T}(M)$ whose non-identity elements are called \emph{twists} (see below for the exact definition). These are analogous to Dehn twists in surface topology. It turns out that there are examples where $\mathcal{T}(M)$ is nontrivial. Thus, $\Out(\pi_1(M))$ can be infinite, unlike the locally symmetric case, for which the outer automorphism group of the fundamental group is finite.

\begin{theorem}[Computation of $\Out(\pi_1(M))$]\label{Out(pi_1(M))}
Let $M$ be a piecewise locally symmetric manifold and the pieces in the decomposition of $M$ are irreducible locally symmetric manifolds. Then the outer automorphism group $\Out(\pi_1(M))$ is an extension of a free abelian group $\mathcal{T}(M)$ by a group $\mathcal{A}(M)$, i.e. the following sequence is exact.
\[1 \longrightarrow \mathcal{T}(M) \longrightarrow \Out(\pi_1(M)) \longrightarrow \mathcal{A}(M) \longrightarrow 1.\] 
If $M$ has finitely many pieces, then $\mathcal{T}(M)$ is finitely generated and $\mathcal{A}(M)$ is finite.
\end{theorem}
Even though Nielsen realization holds for piecewise locally symmetric manifolds, not every such manifold admits a Riemannian metric such that $\Out(\pi_1(M)) \cong \Isom(M)$. This is in contrast to the case of Mostow-Prasad-Margulis rigidity. 
\begin{theorem}\label{weak rigidity}
Let $M$ be a closed, piecewise locally symmetric manifolds of dimension $\geq 3$. If $\Out(\pi_1(M))$ is infinite, then $M$ does not admit any metric such that $\Out(\pi_1(M))$ can be realized as a group of isometries.
\end{theorem}

This paper is organized as follows. We recall the Borel-Serre compactifications of locally symmetric spaces in Section \ref{sec:compactifications}, and the reflection group trick in Section \ref{sec: reflection group}. We then define piecewise locally symmetric manifolds in Section \ref{sec:definitions}. We remark on the lack of $\CAT(0)$ metrics on piecewise locally symmetric manifolds, and raise a question in Section \ref{non CAT(0)}. The complex of groups structure of $\pi_1(M)$ is discussed in Section \ref{pi_1}. We prove Theorem \ref{asphericity} in Section \ref{developability}, and Theorem \ref{piece rigidity} in Section \ref{proof of piece rigidity}. We discuss the structure of $\Out(\pi_1(M))$ and prove Theorem \ref{Out(pi_1(M))} in Section \ref{Out}. Finally we prove Theorem \ref{Mostow} in Section \ref{rigidity}. There is an appendix at the end of the paper that is a summary of theory of complexes of groups that we will need to use in this paper. 

\begin{acknowledgement} I would like to thank my advisor, Benson Farb, for his guidance, directions, constant encouragement, and extensive comments on earlier versions of this paper. I would like to thank Ilya Gekhtman and Dave Witte-Morris for a lot of help on the theory of arithmetic lattices. I would also like to thank Mike Davis and Shmuel Weinberger for useful conversations.
\end{acknowledgement}

From now on, by a \emph{locally symmetric} manifold we will mean a complete, finite volume, locally symmetric, nonpositively curved manifold of noncompact type.

\section{Compactifications of locally symmetric spaces} \label{sec:compactifications}

Let $M$ be a connected, noncompact, finite-volume, complete, locally symmetric, nonpositively curved manifold of noncompact type. Then $M = \Gamma\setminus G/ K$ for some semisimple Lie group $G$ with $K$ a maximal compact subgroup of $G$ and $\Gamma$ is a torsion-free lattice of $G$ that is isomorphic to the fundamental group $\pi_1(M)$. 

If $M$ is an arithmetic manifold, then it has a compactification $\M$ called the Borel-Serre compactification of $M$ (see \cite{BS}). The space $\M$ is a compact manifold with corners whose interior is diffeomorphic to $M$. The corners of $\M$ correspond to  $\Gamma$-conjugacy classes of rational parabolic subgroups of $G$. For example, let $\Gamma$ be a torsion-free finite index subgroup of $\SL(n,\Z)$, and let
\[M = \SO(n) \backslash \SL(n,\R)/ \Gamma.\] 
Then $M$ is a noncompact, finite volume, locally symmetric manifold. Consider the case where $n = 3$. The compactification $\X$ is a manifold with corners with codimension $1$ and codimension $2$ strata. (See also Figure~\ref{SL(3,Z)} for an schematic picture of $\M$.)

\begin{figure}\label{SL(3,Z)}
\begin{center}
\includegraphics[height=100mm]{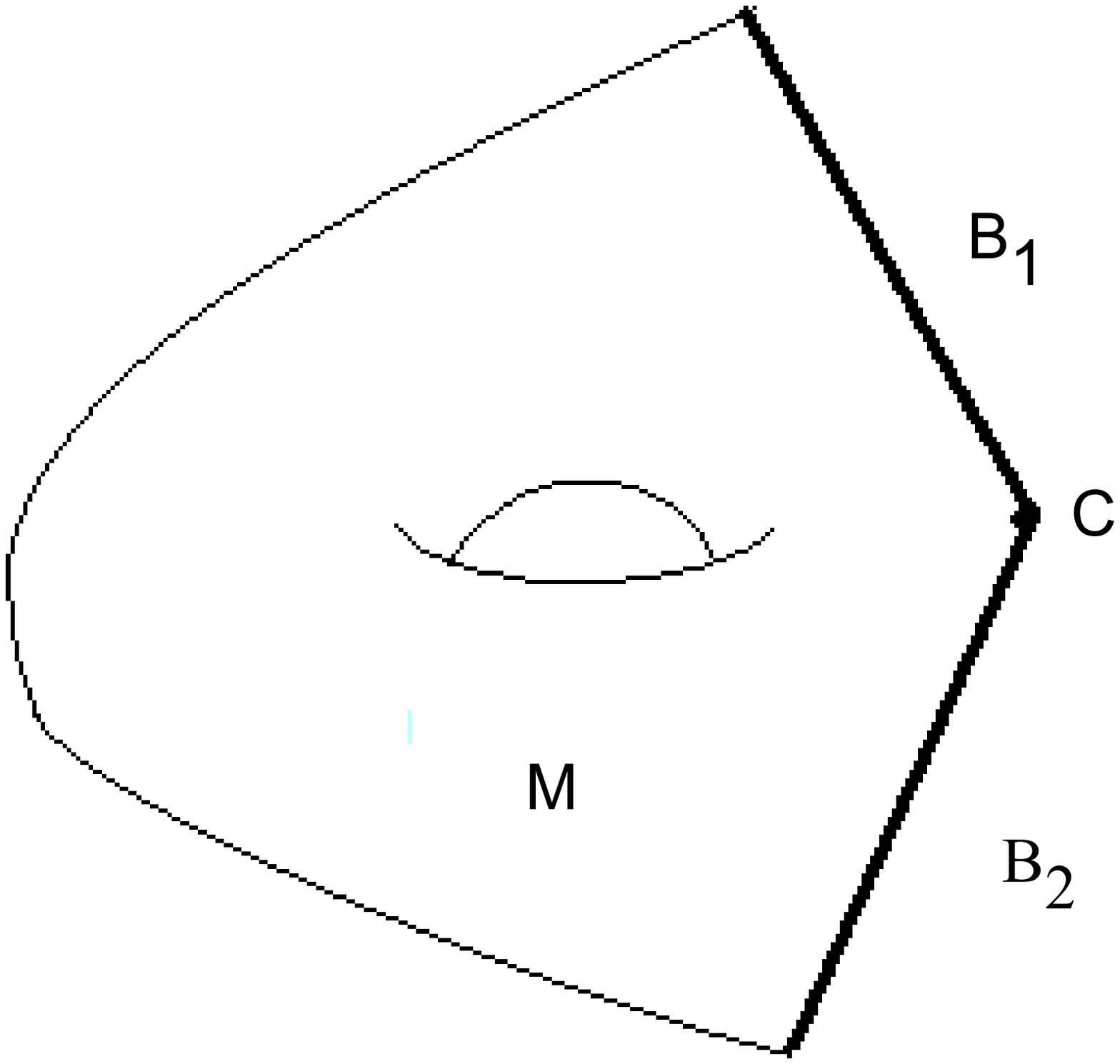}
\caption{Schematic of a manifold with corners with two codimension $1$ strata and one codimension $2$ stratum.}
\end{center}
\end{figure}  

Recall that the parabolic subgroups of $\SL(n,\R)$ are precisely the subgroups that preserve a flag. There are two kinds of codimension $1$ strata, which correspond respectively to the $\Gamma$-conjugacy classes of the following two parabolic subgroups of $\SL(3,\R)$.
\[ P_1 =\left(\begin{matrix} *&*&*\\ *&*&*\\ 0&0&* \end{matrix}\right), \qquad P_2 =\left(\begin{matrix} *&*&*\\ 0&*&*\\ 0&*&* \end{matrix}\right),\]
and one kind of codimension $2$ stratum, which corresponds to the $\Gamma$-conjugacy class of the parabolic subgroup
\[Q =\left(\begin{matrix} *&*&*\\ 0&*&*\\ 0&0&* \end{matrix}\right).\]
For each $i = 1, 2$, the codimension $1$ strata are the following spaces 
\[B_i = (P_i \cap \SO(3))\backslash P_i/ (P_i\cap \Gamma),\] 
which a $2$-torus bundle over $\SO(2)\backslash\SL(2,\R)/(\Gamma \cap\SL(2,\Z))$. The codimension $2$ strata are of the form 
\[C = (Q\cap \SO(3))\backslash Q/ (Q\cap \Gamma),\] 
which is a compact nilmanifold. 

In general, the Borel-Serre compactification of an arithmetic manifold $M$ is constructed as follows. As before, $M = K \backslash G/\Gamma$ for some semisimple, linear, connected algebraic group defined over $\Q$ with a maximal compact subgroup $K$, and $\Gamma$ is an arithmetic lattice of $G$. Let $X = G/K$, which is a symmetric space of noncompact type and the universal cover of $M$. 

Let $P$ be a rational parabolic subgroup of $G$. Let $N_P$ be the unipotent radical of $P$. The Langlands decomposition of $P$ is $P = M_PA_PN_P$, where $N = N_P$, $A_P$ is a $\Q$ split torus, and $M_P$ is a reductive subgroup of $P$ with finitely many components. The groups $M_P$, $A_P$ and $N_P$ are defined over $\Q$, and $M_PA_P = C_G(A_P)$, the centralizer of $A_P$ in $G$. The latter fact together with $N_P$ being normal in $P$ implies that $A_P$ normalizes $M_PN_P$. 

Each parabolic $P$ acts transitively on $X$. The Langlands decomposition of $P$ gives $X$ a \emph{horospherical decomposition}

\[ X = N_P \times A_P \times X_P,\]
where $X_P$ is the symmetric space corresponding to $M_P$, i.e. $X = M_P/(M_P \cap K)$. In this case $A_P$ is isomorphic to the multiplicative group $(0,\infty)^k$, for some positive integer $k$. The Borel-Serre partial compactification is 
\[ \X = \left(X \coprod_{P} X(P)\right) / \sim, \]
where $P$ runs through all rational parabolic subgroups of $G$, and $\sim$ is an equivalence relation defined as follows. For each pair of rational parabolic subgroups $P$ and $Q$, let $R$ be the smallest rational parabolic subgroup containing both $P$ and $Q$. Such $R$ always exists since $G$ is a parabolic subgroup containing $P$ and $Q$. There are embeddings (\cite[Proposition III.5.7]{Ji})
\[i_1 \colon X_R \longrightarrow  X_P, \quad \text{and} \quad i_2 \colon X_R \longrightarrow X_Q .\] 
Then for any $x \in X_R$, the points $i_1(x) \in X_P$ and $i_2(x) \in X_Q$ are defined to be equivalent.

The space $\X$ has the structure of an analytic manifold with corners. For $g \in G_\Q$ , the natural action of $g$ on $X$ extends to $\X$ (\cite[Proposition III.5.13]{Ji}). The quotient $\Gamma \backslash \X$ is a compact analytic manifold with corners, called the \emph{Borel-Serre compactification} of $\Gamma \backslash \X$, which we will denote $\M$. The manifold with corners $\M$ is compact with interior diffeomorphic to $M$.  

By the Margulis arithmeticity theorem \cite[Theorem 6.1.2]{Zimmer}, if the Lie group $G$ has $\R$-rank greater than $1$, then any irreducible lattice of $G$ is arithmetic. Hence, if $M$ is higher rank and irreducible, then $M$ always has a compactification that is a manifold with corners. 
 
If the $\R$-rank of $M$ is $1$, then $M$ may not be arithmetic, in which case $M$ is negatively curved with pinched curvature. Then $M$ has finitely many ends, or \emph{cusps}. Each cusp is diffeomorphic to $[0, \infty) \times S$ for some compact $(n-1)$--dim manifold $S$. The fundamental group of each cusp, or each \emph{cusp subgroup}, corresponds to a maximal subgroup of $\pi_1(M)$ of parabolic isometries fixing a point on the boundary at infinity of $\widetilde{M}$. The parametrization $[0,\infty) \times S$ of a cusp can be taken so that each cross section $a\times S$ is the quotient of a horosphere in $\widetilde{M}$ by the corresponding cusp subgroup. If we delete the $(b, \infty) \times S$ part of each cusp $[0, \infty) \times S$ of $M$, the resulting space is a compact manifold with boundary which we call $\M$. We can identify $M$ with the interior of $\M$. Hence, $\M$ is a compactification of $M$. We choose $b$ large enough so that the boundary components $b \times S$ of different cusps do not intersect. 

If $M$ is reducible, i.e. $M$ is finitely covered by a product $Y = \prod_{j \in J}Y_j$, for some finite collection of noncompact, irreducible, locally symmetric spaces $Y_j$, then $M$ has an obvious compactification as follows. The cover $Y$ has a compactification $\overline{Y} = \prod_{j \in J}\overline{Y_j}$. Since the cover space action on $Y$ extends to $\overline{Y}$, this gives a compactification of $M$ as the quotient of $\overline{Y}$ by the covering space action. 

\section{The reflection group trick}\label{sec: reflection group}
In \cite{Davis}, Davis constructs, for an appropriate Coxeter system $(W,S)$ and a mirrored space $X$, a space $U(W,X)$ with a proper $W$-action. We give a brief overview of this construction. (See \cite{Davis} for more detailed discussion.)

A \emph{mirrored space $X$ over $S$} is a topological space $X$ with a collection of closed subspaces $\{X_s \: | \: s \in S\}$ for some index set $S$. For each $x \in X$, let
\[ S(x) = \{ s \in S \: | \: x \in X_s\}.\]
The \emph{nerve} $N(X)$ of $X$ is a simplicial complex defined as follows. The vertex set of $N(X)$ is the set of $s \in S$. A nonempty subset $T$ of $S$ is a simplex of $N(X)$ if and only if $X_T$ is nonempty. 

Let $(W,S)$ be a Coxeter system, i.e. $W$ is has the following presentation. 
\[ W = \langle s_i \in S \: | \: (s_i)^2 = 1, (s_is_j)^{m_{ij}} = 1\rangle.\]
for some (symmetric with positive integer entry) Coxeter matrix $m_{ij}$. For each subset $T \subseteq S$, define $W_T$ to be the subgroup of $W$ that is generated by all $s\in T$. The \emph{nerve} a Coxeter system $(W,S)$ is denoted by $L(W, S)$ and is simplicial complex whose vertices correspond to $s\in S$, and $T \subset S$  is simplex if $T$ \emph{spherical subset} of $S$. (A spherical subset of $S$ is a subset whose elements generate a finite subgroup.)

Define and equivalence relation on $W\times X$ as $(h,x)\sim(g,y)$ if $x=y$ and $h^{-1}g \in W_{S(x)}$. Then $U(W,X)$ is defined to be to quotient space $(W\times X) /\sim$. It is easy to check that $W$ acts on $U(W,X)$ with quotient $X$. This action is proper if $X$ is Hausdorff and the mirror structure of $X$ is \emph{$W$-finite}, that is, for each $T \subset S$ such that $W_T$ is infinite, the intersection $\bigcap_{s\in T} X_s$ is empty. 

One can think of the space $U(W,X)$ as what they see if they are standing in a room of shape $X$ with real mirrors $X_s$'s. For example, an observer in a cubical room with one mirror on a wall will think that they are in a doubly larger room that is the union of two identical cubical rooms along the wall with the mirror. This corresponds to the case where $X$ is a cube with one of the faces a mirror and $W$ is a one-generator Coxeter system. Another example is when a $2$-dimensional observer is in a triangular room with angles $\pi/m$, $\pi/n$ and $\pi/k$ for some integers $m, n, k > 1$. In the case where $m = n = k = 3$, for example, the observer will think that they are in a flat plane tiled with copies of the room (by looking at the repetition of the same furniture in the room for example). This corresponds to the case where $X$ is a equilateral triangle with mirrors that three edges and $W$ is the $(3,3,3)$ triangle group.

\bigskip
\noindent
\textbf{Manifolds with corners:} A natural class of mirrored spaces is the class of  \emph{nice manifolds with corners}. A \emph{smooth $n$-manifold with corners} $X$ is a second countable Hausdorff space that is differentiably modelled on the standard $\R^n_+ = [0,\infty)^n$.  Given a point $x\in X$ and chart $\phi$ on a neighborhood of $x$, the number $c(x)$ of of coordinates of $\phi(x)$ is independent of the chart $\phi$. For $0 \leq k \leq n$, a connected component of $c^{-1}(k)$ is a \emph{stratum} of $X$ of codimension $k$. The closure of a stratum is \emph{closed stratum}. For each $x \in X$, let $m(x)$ be the set of closed codimension $1$ strata containing $x$. A manifold with corners $X$ is \emph{nice} if the cardinality of $m(x)$ is $2$ for all $x$ with $c(x) = 2$.

A nice manifold with corners $X$ has a natural mirror structure where each closed codimension $1$ stratum is a mirror. If $(W,S)$ is a Coxeter system and $X$ is a nice manifold with corners whose codimension $1$ strata are $C_s$ for $s\in S$ gives $X$ a mirror structure that is $W$-finite, then $U(W,X)$ is a manifold, and $W$ acts on $U(W,X)$ properly and \emph{locally linearly} (\cite[Proposition 10.1.10]{Davis}).

If $X$ is a nice manifold with corners with the natural mirror structure, then the nerve of $X$ is a \emph{convex cell complex}. There is a Coxeter system $(W,S)$ such that $L(W,S)$ is the barycentric subdivision of the nerve $N(X)$ of $X$. For example, for any convex cell complex $\Lambda$, there is a right-angled Coxeter system $(W,S)$ with nerve $b\Lambda$ (\cite[Lemma 7.2.2]{Davis}). For such a Coxeter system, the space $U(W,X)$ is a manifold. One just needs to check that the mirror structure of $X$ is $W$-finite. The space $U(W,X)$ is a differential manifolds with differential structure induced by those on each copies of $X$ in $U(W,X)$.   

The space $U(W,X)$ need not be compact even if $X$ is compact since $W$ can be infinite. However, it has a compact manifold quotient since $W$ acts on $U(W,X)$ properly with compact quotient, and $W$ has a finite index torsion-free subgroup $\Gamma$. The latter is because $W$ admits a faithful representation into some general linear group. Let $M$ be the quotient of $U(W,\X)$ by $\Gamma$. Then $M$ is a compact manifold.

\section{Piecewise locally symmetric spaces: definition and examples}\label{sec:definitions}

Let $X$ be a locally symmetric manifold and let $\overline{X}$ be the Borel-Serre compactification of $X$. Then $\overline{X}$ is a nice manifold with corners (\cite[Proposition III.5.14]{Ji}). 

Let $(W,S)$ be a Coxeter system with nerve $L(W,S)$ the barycentric subdivision of the nerve $N(\overline{X})$ of $\overline{X}$. Then $U(W,\overline{X})$ is a smooth manifold with the smooth structure induced by that of $\overline{X}$. There is a natural action of $W$ on $U(W,\X)$. We call manifolds that are quotients of $U(W,\X)$ by subgroups of $W$ \emph{piecewise locally symmetric spaces}. These manifolds need not be compact, but there are always compact ones, e.g. $U(W,\X)/\Gamma$, for some torsion-free finite index subgroups of $W$. 

We remark that one can similarly define \emph{piecewise locally symmetric manifolds with boundary} by picking $W$ to be a Coxeter group that is generated by a proper subset of the set of mirrors of $\X$.

Now, we give a few examples of piecewise locally symmetric spaces.
\begin{itemize}
\item[a)] The simplest example of a piecewise locally symmetric space is the double of a nonpositively curved, locally symmetric space $X$ with one cusp.  In this case, the Coxeter group $W$ is that with one generator since $X$ has only one boundary component.

\item[b)] Let $X$ be the product of two $\R$-rank $1$ locally symmetric spaces $X_1$ and $X_2$, each of which has one cusp. The boundary $\partial X$ is the union of $B_1 = X_1\times\partial X_2$ and $B_2 = \partial X_2 \times X_1$ along the corner $C = \partial X_1 \times \partial X_2$. So $X$ has two codimension $1$ strata $B_1$ and $B_2$, which are the mirrors of $X$, and one codimension $2$ stratum $C$. Let $W$ be the Coxeter system generated by $B_1$ and $B_2$ with the relation $(B_1 B_2)^k = 1$ for some $k >1$. That is, $W$ is the dihedral group $D_k$. Let $M$ be $U(W,X)$. Then $M$ is a piecewise locally symmetric space. Since $W$ is finite, $M$ is a compact manifold. As we will see in Section \ref{pi_1}, the fundamental group of $M$ has the structure of a $2n$-gon of groups. (See \cite{Bridson}, \cite{stallings}, or the appendix below for the theory of complexes of groups). 

\item[c)] If the spaces $X_1$ and $X_2$ in the above example have more than one cusp, i.e. the compactification $\X_1$ or $\X_2$ has more than one boundary component, then the group $W$ is in general not the reflection group of a polygon. The fundamental group of $M$ has the structure of a $2$-dimension complex of groups.

\item[d)] Let $\Gamma$ be a torsion-free, finite index subgroup of $\SL(n,\Z)$ and let $X$ be $\SO(n) \backslash \SL(n,\R)/ \Gamma$. The Borel-Serre compactification $\X$ is a manifold with corners as we saw in section \ref{sec:compactifications}. Since $\Gamma$ is a subgroup of $\SL(n,\Z)$, the set of mirrors of $X$ may contain more than $(n-1)$ but finitely many mirrors. Let $W$ be the corresponding right angle Coxeter system. The manifold $U(W,X)$ is the union of copies of $X$ glued to each other along pairs of codimension $1$ boundary strata. 
\end{itemize}

Each copy of $\X$ is called a \emph{fundamental chamber} of $M$. Note that all the pieces $M_i$ are diffeomorphic. For each piece $M_i$ of $M$, there is a retraction of $r \colon M \longrightarrow M_i \cong \X$ defined as follows. For each $x \in M$, let $\hat{x}$ be a lift of $x$ in $U(\X, W)$. Then $r(x) = p(\hat{x})$, where $p \colon U(W, \X) \longrightarrow \X$ is the covering projection. 

Two codimension $1$ strata $S_1$ and $S_2$ is said to have the same \emph{type} if  $r(S_1) = r(S_2)$. 

\section{Lack of locally $\CAT(0)$ metrics for piecewise locally symmetric manifolds}\label{non CAT(0)}
Firstly we recall the definition of a $\CAT(\kappa)$ metric space (\cite{Bridson}). Let $X$ be a metric space. Let $M_{\kappa}^2$ be the complete, simply connected, $2$-dimensional manifold with constant curvature equal to $\kappa$. Let $D_{\kappa}$ be equal to  $\pi/\sqrt{\kappa}$ if $\kappa > 0$, and $\infty$ otherwise.  

Let $\triangle$ be a geodesic triangle in $X$ with vertices $p, q , r$, and with perimeter less than $2D_{\kappa}$. A triangle $\overline{\triangle}$ in $M_{\kappa}^2$ with vertices $\bar{p}, \bar{q}, \bar{r}$ is a comparison triangle for $\triangle$ if \[d(p,q) = d(\bar{p}, \bar{q}), \quad d(q,r) = d(\bar{q}, \bar{r}), \quad \text{and} \quad d(r,p) = d(\bar{r}, \bar{p}).\] 
A point $\bar{x}$ on the geodesic segment $[\bar{p},\bar{q}]$ is a comparison point for $x \in [p,q]$ if $d(p,x) = d(\bar{p}, \bar{x})$. Comparison points on $[\bar{q},\bar{r}]$, and $[\bar{r},\bar{p}]$ are defined similarly. 

Such a triangle $\triangle$ is said to satisfy the $\CAT(\kappa)$ inequality if for all $x, y \in\triangle$, and all comparison points $\bar{x}$, $\bar{y} \in \overline{\triangle}$,
\[d(x,y) \leq d(\bar{x},\bar{y}).\]
A geodesic metric space $X$ is \emph{$\CAT(\kappa)$} is all such triangles satisfy the $\CAT(\kappa)$ inequality. A metric space $X$ is said have \emph{curvature} $\leq \kappa$ if it is locally a $\CAT(\kappa)$ space. 

It is a well-known fact that if a Riemannian manifold $M$ has sectional curvature $\leq 0$, then $M$ is locally $\CAT(0)$. 

Piecewise locally symmetric manifolds generally do not admit a locally $\CAT(0)$ metric. For example, if a closed piecewise locally symmetric $M$ is obtained from complex hyperbolic pieces, then $M$ does not admit a locally $\CAT(0)$ metric. More generally, if $M$ is obtain from any locally symmetric manifold whose fundamental group has a solvable, nonabelian subgroup, then $M$ does not admit a locally $\CAT(0)$ metric. This is because of Theorem 7.8 in \cite{Bridson}, which says that any solvable subgroup of a locally $\CAT(0)$ group $\pi_1(M)$ must be virtually abelian. 

However, some piecewise locally symmetric manifold do admit nonpositively curved Riemannian metrics (and thus, locally $\CAT(0)$ metrics). For example, if a piecewise locally symmetric $M$ is obtained from hyperbolic pieces, then one can smooth out the hyperbolic metric (which is singular at gluing) to get a nonpositively curved Riemannian metric. Another example if the manifold in first example of the introduction. One can take the product of the hyperbolic metric on $S_1$ and $S_2$. Gluing these metrics will give a locally $\CAT(0)$, singular Riemannian metric metric on $M$. Nonetheless, on can smooth out this metric to get a nonpositively curved metric on $M$.

If $M$ is obtained from pieces that are products of at least $3$ hyperbolic manifolds with cusps, then $M$ has a locally $\CAT(0)$ metric, which can be taken to be the glued up product hyperbolic metrics from each piece.  
\newline
\newline
\textbf{Question:} If $M$ is as in the previous line, then does $M$ admit a nonpositively curved $C^2$ (respectively, smooth) Riemannian metric?

\section{Fundamental groups and universal covers of piecewise locally symmetric manifolds}\label{pi_1}
Let $M$ be a piecewise locally symmetric manifold of $\R$-rank $\geq 2$, and let $\{M_i\}_{i \in I}$ be the pieces in the  decomposition of $M$. That is, there is a locally symmetric space $X$ with finite volume, a Coxeter system $(W,S)$ and a finite index, torsion-free subgroup $W'$ of $W$ such that $U(W,\X)$ is a manifold and $M = U(W,\X)/W'$. Then  space $U(W,X)$ is a complex of spaces over some complex $\Lambda$ (see \cite{Corson} for the definition of a complex of spaces). The barycentric subdivision of $\Lambda$ is the complex $\Sigma := \Sigma(W,S)$ associated to $(W,S)$ as defined in \cite{Davis}. We briefly recall the definition of $\Sigma$. Let $K$ be the cone on the barycentric division of $L(W,S)$. Then $\Sigma = U(W,K)$. Some important properties of $\Sigma$ are as follows. 

There is a natural cell structure on $\Sigma$ (\cite{Davis}) so that
\begin{enumerate}
\item[a)] its vertex set of $W$, its $1$-skeleton is the Cayley graph of $(W,S)$, and its $2$-skeleton is a Cayley $2$-complex,
\item[b)] each cell is a Coxeter polytope,
\item[c)] the link of each vertex is isomorphic to $L(W,S)$,
\item[d)] a subset of $W$ is the vertex set of a cell if and only if it is a spherical coset of $(W,S)$.
\end{enumerate}

The space $U(W,\X)$ has the structure of a complex of spaces, where the base complex is $\Sigma$ with the cell structure as in the above theorem. (See the appendix below for a short summary of necessary results from the theory of complex of groups and spaces.) The barycentric subdivision of this cell structure is $U(W,K)$. Since $W'$ acts on $U(W,\X)$ preserving fibers, the quotient $M = U(W,\X)/W'$ has the structure of a complex $Y_M$ of spaces, for $Y_M = \Lambda/W'$ since $W'$ is torsion-free. The fiber over each point is a stratum of a copy of $\X$. 


The complex of spaces structure of $M$ gives $\pi_1(M)$ the structure of the fundamental group of a complex of groups $G(T_M)$ over $T_M$ (see the appendix of this paper or \cite{Bridson} and \cite{LT} for complexes of groups). The vertex groups of $G(T_M)$ are the fundamental groups of the pieces $M_i$. The edge groups of $G(T_M)$ are the fundamental groups of the each codimension $1$ stratum of each piece. The $k$-dimensional cell groups are the fundamental groups of each codimension $k$ stratum. The injective homomorphism from $k$-dimensional cell group to a $(k-1)$-dimensional cell group are the obvious inclusion of the corresponding strata. 

\section{Developability of the fundamental group and asphericity}\label{developability}

In general, the vertex groups are not subgroups of a complex of groups \cite{stallings}. Complexes of groups that contain their vertex groups as subgroups are called \emph{developable}. A necessary condition for a complex of groups to be developable is that it is \emph{nonpositively curved} \cite{Bridson}. 

\begin{theorem}[Developability of $\pi_1$]\label{developable}
Let $M$ be a piecewise locally symmetric space. Then the natural complex of groups structure of $\pi_1(M)$ is nonpositively curved and thus is developable.
\end{theorem}


\begin{proof}
Let $M_i$, $i\in I$, be the pieces in the decomposition of $M$. In order to prove that $\pi_1(M)$ is nonpositively curved as a complex $T_M$ of groups, we need to give $T_M$ a metric such that the local development (see \cite{Bridson} or the appendix about local developments) of each vertex in $T_M$ is nonpositively curved. We give $T_M$ a piecewise Euclidean metric as follows. 

Firstly, give $\Sigma(W,S)$ a piecewise Euclidean metric as in \cite[Section 12.1]{Davis}. Since $T_M$ is a quotient of $\Sigma$ by a group of isometries, this metric induces a piecewise Euclidean metric on the polyhedral complex $T_M$. Now consider the barycentric subdivision of $T_M$. There are only finitely many isometry classes of simplices of $T_M$. So the local development $\st(\widetilde{\sigma})$ of each vertex $\sigma$ of $T_M$ is a piecewise Euclidean complex with finite shape. Given this, Theorem 5.2 in \cite{Bridson} says that $\st(\widetilde{\sigma})$ is $\CAT(0)$ if the link at each vertex of $\st(\widetilde{\sigma})$ is $\CAT(1)$.


For each $i$, let $B_i$ be the rational spherical Tits building of $X_i$. Let $\sigma$ be a vertex of $T_M$ that corresponds to some $X_i$. Then the local development $\st(\widetilde{\sigma})$ of $\sigma$ is isomorphic to the cone on $B_i$ as a simplicial complex, with cone point $\widetilde{\sigma}$. 

The complex $\Lk (\widetilde{\sigma}, \st(\widetilde{\sigma}))$ is piecewise spherical and has simplices of size $\geq \pi/2$, i.e. each edge of a simplex has length $\geq \pi/2$ (\cite[Lemma 12.3.1]{Davis}. By a theorem of Moussong (\cite[Lemma I.7.4]{Davis}), $B_i$ is $\CAT(1)$ if and only if it is a metric flag complex (see \cite[Definition I.7.1]{Davis} for the definition of metric flag complexes). But $\Lk (\widetilde{\sigma}, \st(\widetilde{\sigma}))$,which is isomorphic to $B_i$, is a flag complex. Each simplex of $B_i$ corresponds to a spherical subgroup of $W$. For the same reason as in \cite[Proof of Lemma 12.3.1]{Davis}, the piecewise spherical complex $B_i$ is a metric flag complex and thus $\CAT(1)$. Hence $\st(\widetilde{\sigma})$ is $\CAT(0)$.

Let $\zeta$ be a vertex of $T_M$ that corresponds to the barycenter of a (polyhedral) cell of positive dimension. Let $\sigma$ be a vertex of this cell. Then $\st(\widetilde{\zeta})$ is isometric to a neighborhood of point $p$ in $\st(\widetilde{\sigma})$. We can pick $p$ to be in the interior of a cell of type $\zeta$ in $\st(\widetilde{\sigma})$  
Since $\st(\widetilde{\sigma})$ is $\CAT(0)$ as shown above, $\st(\widetilde{\zeta})$ is $\CAT(0)$. 

We have checked that the local development of each vertex in the barycentric subdivision of $T_M$ is $\CAT(0)$. Therefore, the complex of group of $\pi_1(M)$ is nonpositively curved and thus, developable.

\end{proof}

\begin{lemma}\label{centerless pi_1}
The fundamental group $\pi_1(M)$ of a piecewise locally symmetric space $M$ has trivial center.
\end{lemma}

\begin{proof}
As before, the universal cover $\widetilde{M}$ is a complex $\widetilde{T}$ of spaces. Consider the action of $\pi_1(M)$ on $\widetilde{T}$. If $g \in \pi_1(M)$ is in the center of $\pi_1(M)$, then $g$ preserves the set of common fixed points in $\widetilde{T}$ of each vertex group of $\pi_1(M)$ (which contains exactly just the corresponding vertex). Hence, $g$ fixes every vertex of $\widetilde{T}$. Thus, $g = 1$ and $\pi_1(M)$ has trivial center. 
\end{proof}

Now we prove Theorem \ref{asphericity}. 
\begin{proof}[Proof of Theorem \ref{asphericity}]
Since the complex of groups structure of $\pi_1(M)$ is developable, the universal cover $\widetilde{M}$ is a complex of $\widetilde{M}_i$, which is contractible. Hence, $\widetilde{M}$ is homotopy equivalent to its nerve, which is $\CAT(0)$ and thus contractible. Therefore, $\widetilde{M}$ is contractible.  
\end{proof}

\section{Proof of Theorem \ref{piece rigidity} and automorphisms of fundamental groups}\label{proof of piece rigidity}
Let $M$ be a piecewise locally symmetric manifold that satisfies the conditions of Theorem \ref{piece rigidity}, and let $\{M_i\}_{i \in I}$ be the pieces in the decomposition of $M$ into locally symmetric pieces. Since each piece $M_i$ of $M$ is aspherical, we only need to prove the statement on the level of fundamental group. That is, it suffices to prove that following lemma. 

\begin{lemma}\label{pi_1 piece rigidity}
Let $G_v$ be the fundamental group of an irreducible locally symmetric manifold $Y$ of $\Q$-rank $>1$. Let $\phi \colon G_v \longrightarrow \pi_1(M)$ be an injective homomorphism. Then $\phi(G_v)$ is a conjugate of a vertex group of $\pi_1(M)$.
\end{lemma}

Before proving the lemma, we state three results that we are going to use in in the proof of the lemma. 

\begin{theorem}[\cite{Farb}]\label{nilpotent action}
Let $N$ be a finitely generated, torsion-free, nilpotent group acting on a $\CAT{(0)}$ space $Y$ by semi-simple isometries. Then either $N$ has a fixed point or there is an $N$-invariant flat $L$ on which $N$ acts by translations and hence, factoring through an abelian group.
\end{theorem}
\begin{corollary}[\cite{Farb}]
Let $N$ be a finitely generated, torsion-free, nilpotent group acting on a $\CAT{(0)}$ space $Y$ by semi-simple isometries. Then
\begin{itemize}
\item[1)] If $g^m\in [N,N]$ for some $m >0$, then $g$ has a fixed point.
\item[2)] If $N$ is generated by elements each of which has a common fixed point, then $N$ has a global fixed point. 
\end{itemize}
\end{corollary}

The following theorem is due Farb (\cite{Farb}) and  was proved in a more general setting with $\Q$ replaced by and algebraic number field $k$.
\begin{theorem}[\cite{Farb}]\label{gen by nil}
Let $\Gamma$ be an irreducible lattice of $\Q$ rank $r \geq 2$, and let $\Gamma$ act on a $\CAT(0)$ space $Y$. Then exists a collection of subgroups $\mathcal{C} = \{\Gamma_1, \Gamma_2, ..., \Gamma_{r+1}\}$ such that
\begin{enumerate}
\item The groups in $\mathcal{C}$ generate a finite index subgroup of $\Gamma$. 
\item Any proper subset of $\mathcal{C}$ generates a nilpotent subgroup $U$ of $\Gamma$.
\item There exists $m\in \Z^+$ so that for each $\Gamma_i \in \mathcal{C}$, there is a nilpotent group $N <C$ so that $r^m \in [N,N]$ for all $r \in \Gamma_i$. 
\end{enumerate}
\end{theorem} 
As pointed out in \cite{Farb}, it follows from the corollary above that for each such $U$, the set $\Fix(U)$ in $Y$ is nonempty.

\begin{lemma}\label{pre-fundamental lemma}
Let $Z$ be a $\CAT(1)$ piecewise spherical simplicial complex in which all cells have size $\geq \pi/2$. Let $X,Y$ be cells of $Z$, and let $x \in X$, $y\in Y$. If $d(x,y) < \pi/2$, then $X \cap Y \ne \emptyset$, and there is a point $p \in X\cap Y$ such that $\angle_p(x,y) < \pi/2$. 
\end{lemma}

\begin{proof}
Suppose that $X \cap Y  = \emptyset$. Let $\gamma$ be a  unit speed geodesic connecting $x$ and $y$ so that $\gamma (0) = x$ and $y \in \gamma ([0, \pi/2])$. Let $a < 0$ and $b>0$ so that $\gamma([a,b])$ is a connected component of the intersection of $\gamma$ with $X$. Let $v$ be a vertex of $X$, and let $A_v$ be the union of all cells containing $v$. Then $X \subset A_v$. We pick $v$ so that $d(v,\gamma(a)) > d(v, \gamma(b))$. 

Let $B_v$ be the closed $\pi/2$-ball centered at $v$. Then $B_v$ is isometric to the spherical cone $\Cone (\Lk (v, Z))$ on the link of $v$ in $Z$. If $\gamma$ passes through $v$, then $d(v,y) < \pi/2$. Hence $ v \in Y$, and thus $v \in X \cap Y$, which contradicts the above assumption. So $\gamma$ does not pass through $v$. 

For each $t$ such that $\gamma(t) \in A_v$, let $s_t$ be the geodesic segment connecting $v$ passing through $\gamma(t)$ with length $\max (d(v,\gamma(t)), \pi/2)$.  Let $S$ be the surface defined as the union of all such $s_t$. In the same way as in \cite[Proof of Lemma I.6.4]{Davis}, the surface $S$ is  a union of triangles with common vertex $v$ glued together in succession along $\gamma$. Thus we can develop an $S$ along $\gamma$ locally isometric to $\mathbb{S}^2$, i.e. there is a map $f \colon S \longrightarrow \mathbb{S}^2$ that is a local isometry such that $f(v)$ is the North pole. Hence, $f(\gamma)$ is a geodesic in $\mathbb{S}^2$ that misses $f(v)$. Also, $f(S)$ contains the Northern hemisphere of $\mathbb{S}^2$.

The image $f(\gamma)$ cuts inside the region $f(S)$, which contains the Northern hemisphere $N$ of $\mathbb{S}^2$, so $f(\gamma)$ has length $\geq \pi/2$. Let $d$ be such that $f(\gamma (d))$ is where  $f(\gamma)$ exits $f(S)$. It is not hard to see that $\pi/2 > d > b > 0$. Since $d(v,\gamma(a)) > d(v, \gamma(b))$, it follows that $d(f(v),f(\gamma(a))) > d(f(v), f(\gamma(b)))$. Hence $d(f(\gamma(0)), f(\gamma(d))) > \pi/2$ be spherical geometry, which is a contradiction since $\gamma$ has unit speed. Therefore, $X \cap Y \ne \emptyset$.  

We can pick the point $p$ to be $v$. That $\angle_p(x,y) < \pi/2$ also follows from spherical geometry. 
\end{proof}

\begin{lemma}\label{fundamental lemma}
Let $Z$ be a $\CAT(1)$ piecewise spherical simplicial complex in which all cells have size $\geq \pi/2$. Let $X,Y$ be cells of $Z$, and let $x \in X$, $y\in Y$. Let $X'$ (respectively, $Y'$) be a face of $X$ (respectively, $y$) whose interior contains $x$ (respectively, $y$). If $d(x,y) < \pi/2$, then $X'\cap Y'$ contains face $C$ such that $d(x, C) < \pi/2$.
\end{lemma}

\begin{proof}
We induct on the dimension of $Z$. The base case is when $Z$ has dimension $1$, in which case the lemma is obvious. Suppose that $Z$ has dimension $>1$. By Lemma \ref{pre-fundamental lemma}, there is a point $p \in X\cap Y$ such that $\angle_p(x,y) < \pi/2$. Let $u$ (respectively, $v$) be the intersection of the geodesic ray $px$ (respectively, $py$) with the link $\Lk(p,Z)$ of $p$ in $Z$. Since $\angle_p(x,y) < \pi/2$, we have $d_{\Lk(p,Z)}(u,v) < \pi/2$. Note that $\Lk(p,Z)$ is also a $\CAT(1)$ piecewise spherical simplicial complex in which all cells have size $\geq \pi/2$.

Let $U'$ (respectively, $V'$) be a face in $\Lk(p,Z)$ whose interior contains $u$ (respectively, $v$). Apply the induction hypothesis, there is a point $q \in \Lk(p,Z)$ such that  $U'\cap V'$ contains face $D$ such that $d(u, D) < \pi/2$. 

If $d_Z(p,x) < \pi/2$ , then let $C$ be the face in $Z$ that is spanned by $p$ and $D$. Otherwise, let $C$ be the simplex in $Z$ that corresponds to $C$ in $\Lk(p,Z)$. Then $C$ satisfies the condition in the lemma.
\end{proof}

Now we prove Lemma \ref{pi_1 piece rigidity}.

\begin{proof}[Proof of Lemma \ref{pi_1 piece rigidity}]
To show that $\phi(G_v)$ is a conjugate of a vertex group of $\pi_1(M)$, all we need to prove is that the action of $\phi(G_v)$ on the complex $\widetilde{T}$ fixes a vertex. Let $M_k$ be a locally symmetric piece of $M$. By Theorem \ref{gen by nil}, the group $\Gamma = \pi_1(M_k)$ is virtually generated by a collection $\mathcal{C}$ of nilpotent subgroups $N_1$, $N_2$, ..., $N_{r+1}$. We only need $\Gamma$ to be virtually generated by $3$ nilpotent groups that satisfies the conclusion of Theorem \ref{gen by nil}. So we, instead, write $\mathcal{C} = \{\Gamma_1, \Gamma_2, \Gamma_3\}$, for $\Gamma_1 = N_1$, $\Gamma_2 = N_2$, and $\Gamma_3$ is the (nilpotent) group generated by $N_3, N_4, ... N_{r+1}$. Observe that $\Gamma_i$'s satisfy the conclusion of Theorem \ref{gen by nil}. For each $i = 1, 2, 3$, let $F_i$ be the set of points that is fixed by all elements of $\Gamma_i$. 
\[F_i = \Fix(\Gamma_i).\]  Let 
\[W_{ij} = F_i \cap F_j,\]
for $i, j = 1, 2, 3$. By Theorem \ref{gen by nil}, each $W_{ij}$ is nonempty. It is clear that $F_i$'s and $W_{ij}$'s are convex. 

\begin{figure}
\begin{center}
\includegraphics[height=100mm]{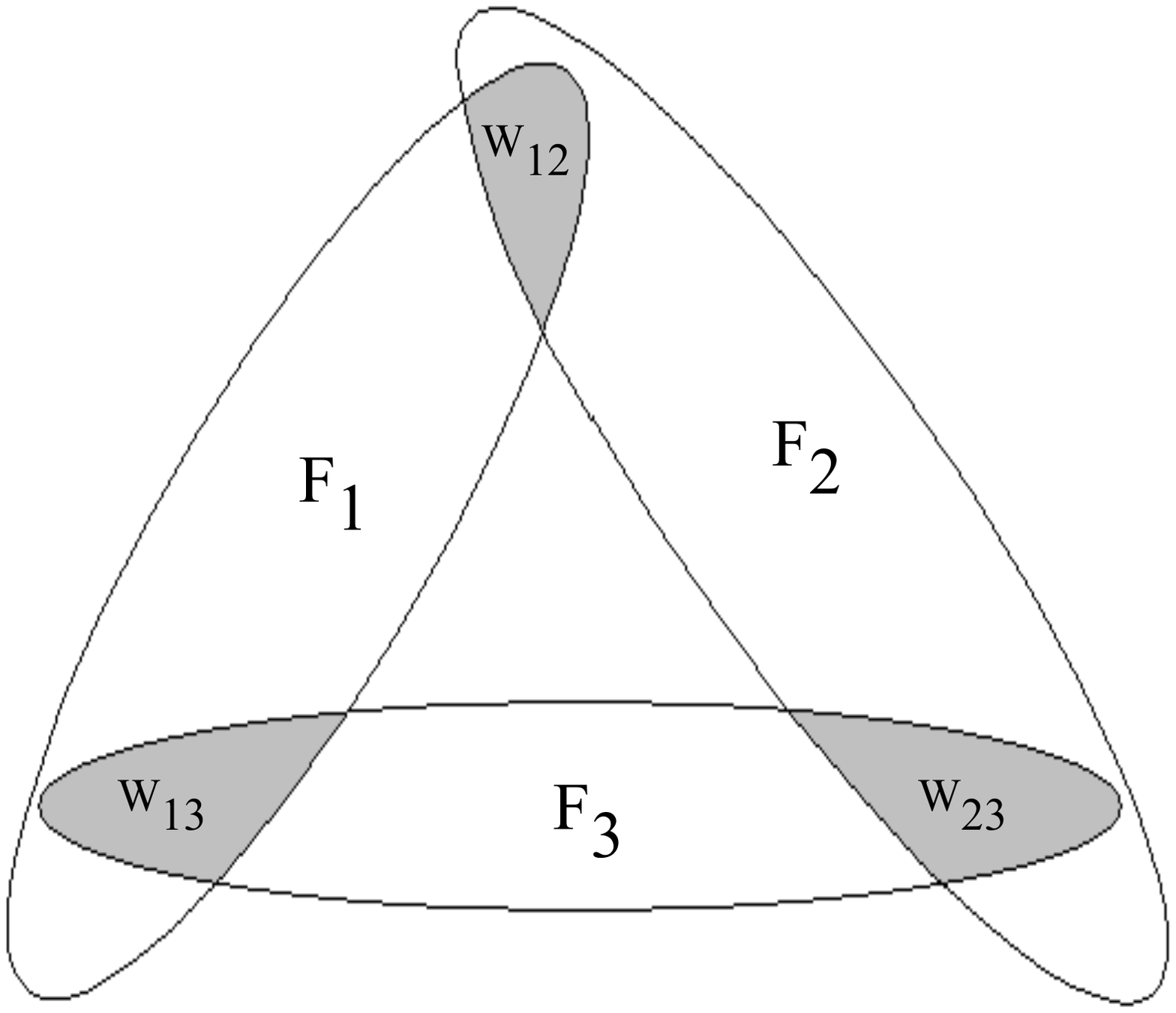}
\caption{The Fix sets $F_1$, $F_2$ and $F_3$ and their pairwise intersections.}
\end{center}
\end{figure}

Suppose that \[\cap_{i=1,2,3}F_i = \emptyset.\]
Let $x_i \in W_{i3}$, for $i = 1,2$, be such that the distance $d(x_1,x_2) = d(W_{13},W_{23})$. Observe that $x_i$ must lie on the boundary of $F_i$, for $i = 1,2$ since $x_1$ and $x_2$ are points realizing the distance between two convex set $W_{13}$ and $W_{23}$. 

Let $y \in W_{12}$. The geodesic $\gamma_{yx_1}$ (and $\gamma_{yx_2}$, respectively) lies in $F_1$ (and $F_2$, respectively) since $F_1$ and $F_2$ are convex. Without loss of generality, suppose that (see \cite{Bridson} for the definition of \emph{angle})
\[\angle_{x_1} (y,x_2) < \pi/2,\] 

We claim that there is a simplex $S$ in $W_{13}$ containing $x_1$ and the angle 
\[\angle_{x_1}(S,x_1x_2) < \pi/2.\]
Given the claim, it follows that $S$ contains some point $x_1'$ other than $x_1$ such that $d(x_1',x_2) < d(x_1,x_2)$, which is a contradiction to the choice of $x_1$ and $x_2$. Therefore,
\[\cap_{i=1,2,3}F_i \ne \emptyset.\] 
That is, $\Gamma$ has a finite index group $\Gamma'$ that fixes a point. Let $H$ be a finite index subgroup of $\Gamma'$ that is normal in $\Gamma$. Then the fix set of $H$ in $L$ is nonempty. Since $\Gamma$ acts on $\Fix(H)$ with bounded orbit, it follows that $\Gamma$ has a global fixed point.

Now we prove the claim. Let $Q$ be the simplex whose interior contains $x_1$. 

Suppose that $Q$ is a point.  Then $Q = \{x_1\}$. Let $L$ be the link of $x_1$. Then $L$ is a $\CAT(1)$ piecewise spherical simplicial complex.  Let $A$ (respectively, $B$) be the cell in $\widetilde{T}$ that intersects nontrivially with $x_1y$ (respectively, $x_1x_2$). Let $A' = \Lk(x_1,A)$ and $B' = \Lk(x_1, B)$. Let $u = L \cap x_1y$ and $v = L \cap x_1x_2$. Then $d(u,v) < \pi/2$ in $L$. Let $U$ (respectively, $V$) be a face of $A'$ (respectively, $B'$) whose interior contains $u$ (respectively, $v$). By Lemma \ref{fundamental lemma}, there is a face $C \subset U\cap V$ such that $d_L(C, v) < \pi/2$.  Let $\alpha$ (respectively, $\beta$) be the span of $x_1$ and $U$ (respectively, $V$). Then $\alpha \subset F_1$ and $\beta \subset F_3$. Let $S$ be the span of $x_1$ and $C$. Therefore, the edge $S \subset W_{13}$ and has angle $< \pi/2$ with $x_1x_2$.  

Suppose that $Q$ is not a point. Then $L:= \Lk(x_1, \widetilde{T})$ is the spherical join $\Lk(x_1, Q)*\Lk(Q,\widetilde{T})$. If $\angle_{x_1}(x_1x_2, Q) < \pi/2$. Then pick $q\in Q$ such that $\angle_{x_1}(x_1x_2, x_1q) < \pi/2$. Since $Q \subset W_{13}$ (because $x_1$ is in the interior of $Q$ and $x_1 \in W_{13}$). We can let $S$ be the edge $x_1q \subset W_{13}$.

Suppose that $\angle_{x_1}(x_1x_2, Q) \geq \pi/2$ (in which case we have equality). Then $x_1x_2$ intersects nontrivially with $\Lk(Q, \widetilde{T})$ at $v$. If $x_1y$ also intersects nontrivially with $\Lk(Q,\widetilde{T})$ at some point $u$, then argue as in the case $Q$ is a point using the fact that $\Lk(Q,\widetilde{T})$ is $\CAT(1)$ piecewise spherical simplicial complex in which all cells have size $\geq \pi/2$  and applying Lemma \ref{fundamental lemma}.

Suppose that $x_1y$ does not intersects nontrivially with $\Lk(Q,\widetilde{T})$. Then there is a point $q \in Q$ such that $x_1y$ intersects with $H := q* \Lk(Q,\widetilde{T})$ nontrivially. Let $u = x_1y \cap H$ and let $v = x_1x_2 \cap H$. Then $d_H(q,v) = \pi/2$ and $d_H(u,v) < \pi/2$. Therefore, the angle $\angle_q (u,v) < \pi/2$. By applying Lemma \ref{fundamental lemma} to the link $\Lk(q, H) = \Lk (Q, \widetilde{T})$, we deduce that there is  $C \subset \Lk(Q,\widetilde{T})$ such that the span $S$ of $Q$ and $C$ is contained in $W_{13}$ and $d_L(p,v) < \pi/2$. Thus $\angle_{x_1} (x_2, S) < \pi/2$.    





\end{proof}
The following theorem is an immediate corollary of Lemma \ref{pi_1 piece rigidity} (for the case where $\M_i$ has $\Q$-rank $>1$). This theorem is true if we drop the $\Q$-rank $> 1$ condition. The case where the pieces $M_i$'s of $M$ have $\Q$-rank $1$ is dealt with in \cite{Tampwrank1}. The case where $M_i$'s are negatively curved, locally symmetric manifolds is dealt with in \cite{Tamrigidity}. So we state it with no assumptions on the $\Q$-rank or rational structure of $M_i$'s.
\begin{theorem}\label{irr vertex group}
Let $M$ be a piecewise locally symmetric space. Assume that each piece in the decomposition of $M$ is irreducible. Let $\phi \colon \pi_1(M) \longrightarrow \pi_1(M)$ be an automorphism. If $G_v$ is a vertex group $\pi_1(M)$, then $\phi(G_v)$ is a conjugate of some vertex subgroup of $M$.
\end{theorem}
Therefore, any automorphism of $\pi_1(M)$ preserves that complex of groups structure of $\pi_1(M)$ in the above sense. 

\section{$\Out(\pi_1(M))$: twists and turns}\label{Out}

If $M$ is finite-volume, complete, irreducible, locally symmetric and nonpositively curved of noncompact type and of dimension $> 2$, then by the Mostow Rigidity Theorem, $\Out(\pi_1(M)) \cong \Isom(M, g_{loc})$, where $g_{loc}$ is the locally symmetric metric on $M$. This implies that $\Out(\pi_1(M))$ is finite since $\Isom(M)$ is finite by Bochner's theorem. One might expect that if $M$ is piecewise locally symmetric, given the Theorem \ref{piece rigidity}, the group $\Out(\pi_1(M))$ will also be finite. However, this is not true, since there can be infinite order homeomorphisms that we call \emph{twists}. (In the $2$-dimensional case, these are Dehn twists in surface topology.)   

To warm up, we give two examples where $\Out(\pi_1(M))$ is infinite and what describe twists are in these cases.

\subsection{Examples of infinite $\Out(\pi_1(M))$.} 
\subsubsection{The pieces $M_i$'s have $\Q$-rank $1$}

Firstly we consider the case of piecewise $\Q$-rank $1$ manifolds. For example, let $M$ be the double of a hyperbolic manifold $N$ with one cusp. Let $M_1$ and $M_2$ be the two copies of $N$ in $M$. Then $\pi_1(M) = G_1*_CG_2$, where $G_i = \pi_1(M_i)$, for $i = 1, 2$, and $C$ is the fundamental group of the cusp. Pick an element $c \ne 1$ in the center of $C$. Let $\phi \colon \pi_1(M) \longrightarrow \pi_1(M)$ be induced by 
\[\phi(g) = 
\begin{cases}
&  g \; \qquad \qquad \text{if} \; g \in G_1 \\
& cgc^{-1} \qquad \text{if} \; g \in G_2.
\end{cases}\]
It is clear that $\phi$ extends to an automorphism of $G_1*_CG_2$ since $\phi$ is an automorphism when restricted to $G_1$ and $G_2$ and agrees on the intersection of $G_1$ and $G_2$. 

\begin{lemma}\label{twist}
Let $M$ and $\phi$ be as above. Then $\phi$ is an infinite order element of $\Out(\pi_1(M))$.
\end{lemma} 

\begin{proof}
For all $k \in \N$, $\phi^k$ is the identity on $G_1$ and conjugation by $c_0^k$ on $G_2$. Suppose $\phi^k$ is an inner automorphism of $G_1*_CG_2$ for some $k \in \N$. Then there exists $g \in G_1*_CG_2$ such that $\Conj(g)\circ\phi$ is the identity on $G_1*_CG_2$. This implies that for $g_1 \in G_1$, we have $g\phi^k(g_1)g^{-1} = g_1$. So $gg_1g^{-1} = g_1$ since $\phi^k$ is the identity on $G_1$. Thus, $g$ is in the centralizer of $G_1$ in $G_1*_CG_2$. 

We claim that $g$ is in the centralizer of $G_1$. Consider the action of $G_1*_CG_2$ on the Bass Serre tree $T$ of $G_1*_CG_2$. The fixed set of $G_1$ is a vertex $v$. Since $g$ commutes with every element of $G_1$, the fixed set of $G_1$ must be preserved by $g$. Hence $g$ belongs to the stabilizer of $v$, which is $G_1$. Therefore, $g$ is in the centralizer of $G_1$ and thus $g =1$. This implies that $\phi$ is the identity homomorphism, which is a contradiction since $\phi$ acts nontrivially on $G_2$. Hence $\phi$ is represents an infinite order element in $\Out(\pi_1(M))$.
\end{proof}

The automorphism $\phi$ of Lemma \ref{twist} is analogous to, on the level of fundamental groups, a Dehn twist in surface topology. For each loop $c$ in the center of $C$, we define a \emph{twist} around a loop $c$ to be an automorphism constructed as above. Up to conjugation, it is the identity on one vertex subgroup and conjugation by an element in the center of the edge subgroup on the other vertex subgroup. By Lemma~\ref{twist}, twists induce infinite order elements of $\Out(\pi_1(M))$. 

\subsubsection{The pieces $M_i$'s have $\Q$-rank $2$}
Although we have been focusing on the case where the $M_i$'s are irreducible, for this example we choose $M_i$ to be a product of two $\Q$-rank $1$ locally symmetric manifolds, since in this case it is easier to see the main features of the proof for the general case. 

Let $M$ be the manifold similar to the example given in the introduction. That is, $M$ is obtained by using $D_3$ to glue six copies of $S_1\times S_2$. We take $S_j$, for $j =1, 2$, to be the usual compactification of a finite volume, hyperbolic manifolds $Y_j$ (rather than just surfaces) with one cusp whose cross section that is a torus. 

Like in the $\Q$-rank $1$ case, $M$ has \emph{twists}, which comes from the center of the fundamental group of each codimension $1$ stratum of $M$. The following is an example.

For each $M_i$, for $i = 1, 2,..., 6$, we denote by $\partial_j M_i$ the boundary stratum the corresponds to $S_j\times\partial S_{j+1}$ (where addition is mod $2$). Note that the center of $\pi_1(S_j\times\partial S_{j+1}) \cong \pi_1(S_j)\times\pi_1(\partial S_{j+1})$ is $\pi_1(\partial S_{j+1})$. 

Without loss of generality, suppose that $\partial_2M_1$ is glued to $\partial_2M_2$. Let $c$ be an element in the center of the fundamental group of $\partial_1M_1$. Let $\phi$ be the homomorphism $\pi_1(M) \longrightarrow \pi_1(M)$ defined as follows.
\[\phi(g) = 
\begin{cases}
&  g \; \qquad \qquad \text{if} \; g \in \pi_1(M_i) \;\text{for $i = 3, 4, 5, 6$}, \\
& cgc^{-1} \qquad \; \text{if} \; g \in \pi_1(M_i) \;\text{for $i = 1, 2$} .
\end{cases}\] 
We see that $\phi$ is a homomorphism since it is a homomorphism on each piece and agrees on the intersections of the pieces. 
\begin{lemma}
Let $M$ and $\phi$ be as above. Then $\phi$ has infinite order in $\Out(\pi_1(M))$.
\end{lemma}
\begin{proof}
The proof is exactly the same as that of Lemma \ref{twist}.
\end{proof}
The main feature of the construction of the above twist is the following. Firstly we pick a codimension $1$ stratum $\partial_1M_1$ of a piece $M_1$ with center containing a nontrivial element $c$. Then we define $\phi$ to be by conjugation by $c$ when restricted to $\pi_1 (M_1)$. So the restriction of $\phi$ to $\pi_1(\partial_1M_1)$ is the identity, but $\phi$ restricted to $\pi_1 (\partial_jM_1)$ for $j \ne 1$ is not. In this case $j = 2$ only. So we need to define $\phi$ on pieces adjacent to $M_1$ along $\partial_j M_1$ (for $j \ne 1$) to be conjugation by $c$. Then we define $\phi$ on the pieces adjacent to these accordingly and so on.

\subsection{The structure of $\Out(\pi_1(M))$}
We now give the set-up in which Theorem \ref{Out(pi_1(M))} is proved. 

Let $\mathcal{M}$ be the disjoint union of all the complete locally symmetric spaces corresponding to the locally symmetric pieces $M_i$ in the decomposition of $M$. By Theorem~\ref{irr vertex group}, an element of $\Out(\pi_1(M))$ has to restrict to an isomorphism between vertex groups up to conjugation, which, by the Mostow Rigidity Theorem, is induced by an isometry with respect to the complete, locally symmetric metric on each of the pieces. The restriction map is a homomorphism from $\Out(\pi_1(M))$ to $\Isom(\mathcal{M})$. We call this induced map 
\[\eta \colon \Out(\pi_1(M))\longrightarrow \Isom(\mathcal{M}).\]
Let $\mathcal{A}(M) $ be the image of $\Out(\pi_1(M))$ under $\eta$. We call the elements of $\mathcal{A}(M)$ \emph{turns}. Then $\mathcal{A}(M)$ is the subgroup of $\Isom(\mathcal{M})$ of isometries of $\mathcal{M}$ whose restriction to the boundaries of each pair of strata that are identified in the decomposition of $M$ are homotopic with respect to the gluing. 

An isometry of $\mathcal{M}$ permutes the pieces $\mathcal{M}_i$'s in the decomposition of $\mathcal{M}$. Thus, for the case where $M$ is of finite type and the decomposition of $M$ has $k$ pieces, there is a homomorphism $\varphi$ from $\Isom(\mathcal{M})$ to the group of permutations of $k$ letters. Let $P$ be the image of $\varphi$. The kernel of $\varphi$ contains precisely those isometries of $\mathcal{M}$ that preserve each of the components of $\mathcal{M}$. Hence, $\Isom(\mathcal{M})$ has the structure of an extension of groups as follows. 
\[ 1 \longrightarrow \bigoplus_{i = 1}^k\Isom(\mathcal{M}_i) \longrightarrow \Isom(\mathcal{M}) \longrightarrow P \longrightarrow 1.\]
It follows that $\Isom(\mathcal{M})$ is finite since it has a finite index subgroup that is isomorphic to the direct sum of the isometry groups of the components of $\mathcal{M}$, which are finite. Thus, $\mathcal{A}(M)$ is finite if $M$ is of finite type. 

We define the \emph{group of twists} of $M$ to be $\mathcal{T}(M) = \Ker\eta$. Then,

\[1 \longrightarrow \mathcal{T}(M) \longrightarrow \Out(\pi_1(M)) \longrightarrow \mathcal{A}(M) \longrightarrow 1.\] 

The group $\mathcal{T}(M)$ is generated by elements called \emph{twists}, which are automorphisms of $\pi_1(M)$ of the following form. The fundamental group $\pi_1(M)$ has the structure of a complex of groups $G(X)$. Suppose that an edge group $G_e$ has nontrivial center. Let $c \in Z(G_e)$ be a non-identity element. Let $v$ be an end point of $e$. A piece $u$ is \emph{$\hat{e}$-adjacent} to $v$ if $u$ can be connected to $v_1$ by a finite sequence of chamber $s_1, s_2, ... s_k$ such that $u = s_1$, $v = s_k$, and for each $i$, the chamber $s_i$ and $s_{i+1}$ and the codimension $1$ they share is adjacent to the edge NOT of type $e$. The \emph{$(e, v)$-block} is the set of pieces that is $\hat{e}$-adjacent to $v_1$.  Let $\{u_j\}$ be the set of vertices the correspond to the $(e, v)$-block.  

Let $\phi$ be an automorphism of $\pi_1(M)$ that is obtained by gluing the identity automorphism of each vertex groups except for $G_{u_j}$'s, on which the automorphism to glue is conjugation by $c$. For the same reasons as in the proof of Lemma \ref{twist}, the map $\phi$ is an automorphism that is either trivial or of infinite order. If the $(e,v)$ block is a proper subset of the set of pieces $M_i$, then $\phi$ is nontrivial and thus, has infinite order. We call such an automorphism a \emph{twist around the loop $c$}. At this point, we have only defined twists on the level of fundamental groups. We will see later that twists can be realized as homeomorphisms.       

For each edge $e$ such that $c$ belong to the center $Z(G_e)$, if $v$ is an end point of $e$, then there is a twist along $c$ in the $(e,v)$-block. Since each twist happens in the union of the chamber in a block, we call the union of all the blocks with nontrivial twists a \emph{twist region}.

Now we prove Theorem~\ref{Out(pi_1(M))}.

\begin{proof}[Proof of Theorem \ref{Out(pi_1(M))}]
We need to show that $\mathcal{T}(M)$ is abelian. To show this, we show that $\mathcal{T}(M)$ is generated by twists and that twists commute.

By definition, the action of a twist on the fundamental group of each of the pieces in the decomposition of $M$ is the identity map up to conjugation by an element in the center of a cusp subgroup. By Mostow rigidity, the image of a twist under $\eta$ is the identity isometry.  Thus, twists belong to $\mathcal{T}(M) = \Ker\eta$. 

Now if $\phi \in \Ker\eta$, then the isometry that $\phi$ induces on $\mathcal{M}$ is the identity map. Hence, for each $i$, the restriction of $\phi$ on $\pi_1(M_i)$ is the identity map in $\Out(\pi_1(M_i))$. Let $\phi_{M_i}$ be the restriction of $\phi$ to $\pi_1(M_i)$. If $M_i$ and $M_j$ are adjacent pieces that are glued together along $S$, then the restriction of $\phi_{M_i}$ to $\pi_1(S)$ is equal to that of $\phi_{M_j}$. This implies that $\phi_{M_i}$ and $\phi_{M_j}$ differ by a conjugation by an element $c$ in the center of $\pi_1(S)$. 

Without lost of generality, assume that $\phi_{M_a}$ is the identity automorphism of $\pi_1(M_a)$, for some $a$. Suppose that $M_b$ is adjacent to $M_a$ along $S$. Let $\tau_{c_1}$ be such that $\phi_{M_b}$ is conjugation by ${c_1}^{-1}$. Let $\tau_{c_1}$ be the $(S, M_b)$-twist along $c_1$. Let $\phi_1 = \tau_{c_1}\circ\phi$. Then $\phi_1 \in \mathcal{T}(M)$ and the restriction of which to the fundamental groups of $M_a$ and $M_b$ is the identity automorphism. One just needs to see that $M_a$ does not belong to the $(S, M_b)$-block, which is true because otherwise, $\phi_{M_a}$ and $\phi_{M_b}$ cannot differ by a conjugation by elements in $\pi_1(S)$ in the first place. 

Now, we repeat this process for another piece that is adjacent to one of the pieces to which the restriction of $\phi_1$ is the identity automorphism to get $\phi_2$ and so on. It is not hard to see that for each $i$, the automorphism $\phi_i \in \mathcal{T}(M)$. Also, the twist region of $\tau_{c_i}$ does not contain any pieces to which the restriction of $\phi_{i-1}$ is the identity automorphism. This process terminates at some point if $M$ has finitely many pieces. Therefore, $\phi$ is a product of twists. 

Since twists are defined for each element $c$ in the center of an edge group $G_e$, it follows that such an element $c$ is also in the center of the highest dimensional cell group that is contained $G_e$. It is not hard to see that any two twists commute since either they are twists around loops in disjoint twist regions or the loops they twist around commute since they are both in the center of the same highest dimensional cell group. 

Hence, $\mathcal{T}(M)$ is a torsion-free abelian group. If $M$ has finitely many pieces, then $\mathcal{T}(M)$ is finitely generated. 
\end{proof}
Topologically, Theorem~\ref{Out(pi_1(M))} says that an element $\Out(\pi_1(M))$ is a composition of twists and turns (on the level of fundamental groups). In the next section, we will prove that one can realize these by actual homeomorphisms of $M$.

\section{Rigidity of piecewise locally symmetric spaces}\label{rigidity}

Before proving Theorem \ref{Mostow}, we state and prove the properties of the boundary strata of the Borel-Serre compactification of noncompact arithmetic manifolds that we will need.

\begin{enumerate}
\item[a)] The cross section of each stratum of an arithmetic manifold $M$ is finitely covered by a manifold $\widehat{C}$ that has the structure of a fiber bundle with fiber a compact nilmanifold $F$ and base a locally symmetric manifold $B$ with nonpositive curvature that can contain a local Euclidean factor. The fundamental group of $\widehat{C}$ has the structure given by the following extension
\[ 1 \longrightarrow \pi_1(F) \longrightarrow \pi_1(\widehat{C}) \longrightarrow \pi_1(B) \longrightarrow 1, \] 
where $\pi_1(F) = \Gamma_{N_P}$, and $\pi_1(B)$ is the fundamental group of a compact, nonpositively curved, locally symmetric space. By the Godement compactness criterion (Theorem $5.30$ in \cite{Witte}), the group $\pi_1(F)$ is nontrivial. 

\item[b)] The center of $M_PN_P$ is contained in $N_P$. 
\end{enumerate}

Let $P$ be a parabolic subgroup of $G$ as in Section \ref{sec:compactifications}. Let $L_P = P/N_P$ be the Levi quotient of $P$. Then both $N_P$ and $L_p$ are rational algebraic groups. The group $L_P \cong A_PM_P$. 
Let 
\[\Gamma_P = \Gamma \cap P \quad \text{and} \quad \Gamma_{N_P} = \Gamma \cap N_P.\] 
Then $\Gamma_P$ is an arithmetic subgroup of $N_P$ and $N_P/\Gamma_{N_P}$ is compact (\cite{Witte}). The projection of $\Gamma_{P}$ in $L_P$ is an arithmetic group of $M_P$ (\cite{BJloc}) denoted by $\Gamma_{M_P}$. Then $\Gamma$ has the structure given by the short exact sequence 

\[ 1 \longrightarrow \Gamma_{N_P} \longrightarrow \Gamma_P \longrightarrow \Gamma_{M_P} \longrightarrow 1 \]
that need not split. The group $\Gamma_{N_P}$ is the fundamental group of a compact nilmanifold and $\Gamma_{M_P}$ is the fundamental group of a compact locally symmetric space with nonpositive curvature.

\begin{proof}[Proof of property (a)]
This follows from the Langlands decomposition of $\Q$ parabolic subgroups of $G$. To see this, let 
$\Lambda$ be a finite index, torsion-free subgroup of $\Gamma_{M_P}$. Then $\Gamma_{N_P}\rtimes \Lambda$ corresponds to a finite-sheeted cover $\widehat{C}$ of $C$ that has the structure of a fiber bundle with fiber $F = N_P/\Gamma_{N_P}$ and base $B = X_P/\Lambda$. The group $\Gamma_{N_P}$ is a cocompact lattice in $N_P$ (see \cite{Witte}). The manifold $X_P/\Lambda$ is a compact, locally symmetric, nonpositively curved manifold since $\Lambda$ is torsion-free. 
\end{proof}

\begin{proof}[Proof of Property (b)]
Firstly, we show that any element $g$ in the identity component $Z^0$ of the center $Z$ of $M_PN_P$ that is not in $N_P$ must commute with the Borel subgroup of $G$. This will be a contradiction since the centralizers of Borel subgroups are trivial. Then we will show that $Z$ must be connected.

Let $T$ be a maximal $\R$-split torus and let $\mathfrak{a}$ be the Lie algebra of $T$.  Then the Lie algebra of the unipotent radical $N_P$ of $P$ is 
\[ \mathfrak{n}_I = \sum_{\alpha \in \Phi^+ \setminus \Phi_I} g_\alpha,  \]
where $\Phi^+$ is the set of positive roots and $\Phi_I \subset \Phi^+$ is a set of simple roots. Let
\[\mathfrak{a}_I = \cap_{\alpha\in I}\ker \alpha. \]
Let $\mathfrak{a}^I$ be the orthogonal complement of $\mathfrak{a}_I$. Then the Lie algebra of $M_P$ is
\[ \mathfrak{m}_I = \mathfrak{a}^I \oplus \sum_{\alpha \in \Phi_I} g_\alpha \oplus l,\]  
for $l$ a subspace of the centralizer of $\mathfrak{a}$. (See \cite{BJloc} for details of the above). Now, since $T$ normalizes $M_PN_P$, it normalizes $Z$. Hence, the Lie algebra $\mathfrak{z}$ of $Z^0$ is a direct sum of subspaces of root spaces $g_\alpha$. If $\mathfrak{z}$ contains elements $x$ in a root space $g_\beta$ not in $\Phi_I$. Observe that $\beta$ cannot be $0$, for $g_\beta$ commutes with $g_\alpha$ for all $\alpha \in \Phi_I$, which means that $g_\beta \subset \mathfrak{a}_I$, which intersects with $\mathfrak{m}_I$ trivially.

If $\beta \in \Phi_I$, then $x$ commutes  with $\mathfrak{n}\oplus\mathfrak{m}$ and $\mathfrak{a}_I$ since $\mathfrak{a}_I \subset \ker g_\beta$. Therefore, $x$ commutes with $\Phi^+  \oplus \mathfrak{a}^I$. Thus  $x$ commutes with a Borel subgroup contained in $P$. Since the centralizer of any Borel subgroup is trivial, this is a contradiction. So $\mathfrak{z}$ is contained in $\mathfrak{n}_I$. Hence $Z^0$ is contained in the unipotent radical $N_P$ of $P$.

If $g_\beta \subset l$, then $x$ commutes with $\mathfrak{a}$ (since $l$ is a subspace of the centralizer of $\mathfrak{a}_I$) and $\Phi^+$. Hence $x$ commutes with the Borel subgroup and again, we get a contradiction. 

Now, we prove that $Z$ is connected. Suppose it is not. Then there is an element $g$ in $Z$ such that $g = mn$, for $n \in N_P$ and $m \in M_P -\{1\}$. Let $X_P$ be the symmetric space $(K\cap M_P)\backslash M_P$. Then $m$ acts on $X_P$, and the action of $m$ commutes with $M_P(\Z)$, which is a lattice of $M_P$. Thus $m$ commutes with the isometry group of $X_P$. 

Since $X_P$ is isometrically a product of symmetric space of noncompact type $Y$ and a Euclidean space $E$, it follows that $m$ must split into a product of two isometries, each of which preserves either $Y$ or $E$. Hence, $m$ acts by a translation on $E$ and by the identity on $Y$. Since each translation can be connected to the identity via a one parameter subgroup (namely translation with the same direction but different displacement), this means that if the above translation is nontrivial, then there is vector in $\mathfrak{z}$ that is not in $\mathfrak{n}_I$, which is a contradiction to above. Hence, $m$ acts on $X_P$ trivially, which implies that $\exp(m) \in K\cap M_P$. This means that $m$ is an isometry of $X_P$ the preserves the fiber $N_P$. Thus, the restriction of $m$ to each fiber has to be an isometry.

By \cite[Theorem 9]{Tamrigidity}, any isometry of a simply connected, $2$-step nilpotent Lie group that commute with a lattice has to be left multiplication by an element of the center. The proof of the theorem can be extended easily to the general case of a simply connected nilpotent Lie group. Hence, the element $m$ has to be in the center of $N_P$, which is a contradiction to the assumption that $m \in M_P$. Therefore, $Z^0$ contains elements outside $N_P$, which is a contradiction.   
\end{proof}

Now we prove Theorem \ref{Mostow}.

\begin{proof}[Proof of Theorem~\ref{Mostow}]
Let $\phi \colon \pi_1(M) \longrightarrow \pi_1(M)$ be an isomorphism.  We are going to construct a homeomorphism from $M$ to $M$ by firstly defining homeomorphisms on each piece in the decomposition of $M$ and then gluing them together in such a way that the resulting homeomorphism induces the isomorphism $\phi$. 

Let $M_i$ , for $i$ in some index set $I$, be the pieces in the decomposition of $M$. By Lemma~\ref{irr vertex group}, the map $\phi$ defines a bijection $\alpha \colon I \longrightarrow I$ such that $\phi(\pi_1(M_i)) = \pi_1(M_{\alpha(i)})$ up to conjugation. By the Mostow Rigidity Theorem, the restriction of $\phi$ on each vertex group of $\pi_1(M)$ up to conjugation is induced by an isometry $f_i$ from $M_i$ to $M_{\alpha(i)}$. For each $i$, the isometry $f_i$ extends to a homeomorphism between manifolds with corners from $M_i$ to $M_{\alpha(i)}$ (\cite[Proposition III.5.13]{Ji}), which we will also call $f_i$. 

At this point one may want to glue the isometries $f_i$'s together and claim that it is the desired homeomorphism of $M$. However, there are two problems. One is that the gluing of $f_i$ at each pair of boundaries of $M_i$'s that are glued together might not be compatible with the gluing. The other problem is that if $M_i$ and $M_j$ are adjacent pieces (that is, they share a codimension $1$ stratum), it may happen that ${f_i}_*$ and ${f_j}_*$ define different isomorphisms (by a conjugate) on the fundamental group of $M_i$ and $M_j$ even though they agree on the common boundary stratum. See Lemma~\ref{twist} above for such an example of $\phi$.

Hence we need to do some modifications to $f_i$'s near the boundary of $M_i$ before gluing them together. Specifically, it suffices to show the following two things. Firstly, if $M_i$ and $M_j$ are adjacent pieces whose intersection is $S$, then there is an isotopy from $f_i$ (restricted to $S$) to $f_j$ (restricted to $S$). Once we have this we can modify the maps $f_i$ and $f_j$ by an isotopy of $S$ on a tubular neighborhood of $S$ that is compatible with the gluing of $M$. However, the induced map of the obtained homeomorphism may be different from $\phi$ by a twist. Thus secondly, we need to realize twists as homeomorphisms by realizing the change of basepoint given by the conjugation difference between ${f_i}_*$ and the restriction of $\phi$ to the given stratum. It suffices to show that or a twist along a loop $c$ in the center of $\pi_1(S)$ for some codimension $1$ stratum $S$, there is an isotopy of $S$ that moves the basepoint around a loop homotopic to $c$. So the two problems boil down to proving the same thing, that is, the following theorem, which we will assume for now and give the proof after this. 

\begin{theorem}\label{cusp isotopy}
Let $M$ be a piecewise locally symmetric manifold that has  one stratum $S$. If $f \colon \widetilde{S} \longrightarrow \widetilde{S}$ is an $\pi_1(S)$-equivariant isometry that is $\pi_1(S)$-equivariantly homotopic to the identity map of $\widetilde{S}$, then $f$ is $\pi_1(S)$-equivariantly isotopic to the identity map. Thus, if $g \in Z(\pi_1(S))$, then there is an isotopy on $S$ that moves the base point around a loop that is homotopic to $g$. 
\end{theorem}

Assuming Theorem \ref{cusp isotopy}, we observe that the restriction of $f_i$ and $f_j$ to each pair of boundary components of $M_i$ and $M_j$ that are identified, say $S$, induces the same isomorphism of $\pi_1(S)$ up to conjugation. Since $S$ is aspherical, the restriction of $f_i$ and $f_j$ to $S$ are homotopic isometries. By Theorem~\ref{cusp isotopy}, we can modify the maps $f_i$ and $f_j$ by an isotopy of $S$ on a tubular neighborhood of the corresponding boundary component of $M_i$ and $M_j$ that is compatible with the gluing of $M$ and the action of $\phi$ on fundamental groups. Let $f \colon M \longrightarrow N$ be the homeomorphisms obtained by gluing the modified $f_i$. Then $f_* = \phi$.


Now we construct the homomorphism 
\[\rho \colon \Out(\pi_1(M)) \longrightarrow \Homeo(M)/\Homeo_0(M).\] 
To do this, we need to slightly modify the decomposition of $M$ into pieces. Instead of taking $M$ as a union of compact manifolds with corners $M_i$ glued along pairs of diffeomorphic strata, for each codimension $k$ stratum $P$ of an $M_i$, we glue $P\times [0,1]^k$ to $P$ along $P \times (0,0, ..., 0)$. If $Q \subseteq P$ is codimension $k+1$ stratum, then we glue $Q \times [0,1] \times (0,0, ..., 0)$ to $P \times (0,0, ..., 0)$. The space $N_i$ we obtain is a compact manifold with corners that is diffeomorphic to $M_i$. Hence $M$ is a union of $N_i$ glued to each other in the same way as the $M_i$'s are. We will take this decomposition of $M$ into $N_i$'s for the rest of the proof. 

For each $\varphi \in \Out(\pi_1(M))$, we define $\rho(\varphi)$ to be the homeomorphism that takes $M_i$ isometrically onto $M_{\varphi(i)}$. Any twist and turn will happen on the tubes connecting the boundary strata as follows. If $S$ is a codimension $k$ stratum, one takes the product of the corresponding $k$ straight line isotopies as in the proof of Theorem~\ref{cusp isotopy} that realizes half of the amount that needs to be twisted. 

We claim that $\rho$ is a homomorphism. To see this, for each $\varphi \in \Out(\pi_1(M))$, the homeomorphism $\rho(\varphi)$ is the unique homeomorphism that induces $\varphi$ that is an isometry when restricted to each $M_i$, and that is the straight line isotopy on each tube. The composition of any two such maps is a map of the same type. By uniqueness of such maps, $\rho$ is a homomorphism.  
\end{proof}

\begin{proof}[Proof of Theorem \ref{cusp isotopy}]
We have, as before 
\[ 1 \longrightarrow F \longrightarrow C \longrightarrow B \longrightarrow 1,\]
where $F$ is the torsion-free, discrete, cocompact subgroup of the isometry group of a simply connected Riemannian nilpotent Lie group $H$ (i.e. $H$ is endowed with a left-invariant Riemannian metric), which is the unipotent radical of the parabolic $P$ corresponding to $C$. 

If $f$ is an isometry of $S$ that is homotopic to the identity map, then if $\widetilde{f}$ is a lift of $f$ to the universal cover $\widetilde{S}$ of $S$, then $\widetilde{f}$ commutes with all the deck transformation of $\widetilde{S}$. By the proof of Property (b) above, the map $\widetilde{f}$ is in the center of the unipotent radical $N_P$ of $P$.

Since, $N_P$ is a nilpotent, simply connected Lie group, the exponential map is a global diffeomorphism from $\mathfrak{n}_P$ to $N_P$. In logarithmic coordinates of $N_P$, one can take the straight line in $N_P$ from the identity to $\widetilde{f}$. 
Since $\widetilde{f}_t$ is a straight line path in the Lie algebra of the nilpotent Lie group $N_P$, and $N_P$ is normal in $P$, it follows that conjugation by $\gamma \in C$ preserves this path. Since $\widetilde{f}_1 = \widetilde{f}$ commutes with $C$, it follows that the same is try for $\widetilde{f}_t$ for a given $t \in [0,1]$. That is, 
\[ \widetilde{f}_t \circ \gamma = \gamma \circ \widetilde{f}_t, \qquad \text{for all} \; \gamma \in C.\]
So multiplication by $\widetilde{f}_t$ in $M_PN_P$ is an $C$-equivariant isotopy. Hence, it descends to an isotopy on $S$.

\end{proof}

\section{Appendix: Complexes of groups}

The theory of complexes of groups was developed as a natural generalization of the Bass-Serre theory of graphs of groups. There are, in general, two perspectives one can take in Bass-Serre theory. One is that in the spirit of van Kampen's theorem. The other is to view it as an inverse problem of groups acting on trees. The theory of complexes of groups seems to have been developed with the latter point of view. 

For the manifolds in this paper, i.e., piecewise locally symmetric spaces, the first point of view is more natural and geometric. In this section, we are not going to develop the whole theory of complexes of groups from scratch, but we are going to present the first perspective for the complexes of groups as a guiding principle of geometric understanding. We are going to review the theory of complexes of groups in the context of this paper. Although this may not be of the greatest generality as the theory in the literature, it gives all that is needed in this paper, and the author believes that this captures the heart of complexes of groups by looking at concrete, simpler settings.

A \emph{convex polytope} $P$ in an affine space $\mathcal{A}$ is the convex hull of a finite subset. The \emph{dimension} of $P$ is the dimension of the affine subspace that it spans.

A \emph{convex cell complex} $Y$ is a collection $\Lambda$ of convex polytopes in an affine space $\mathcal{A}$ such that
\begin{itemize}
\item[i)] if $P \in \Lambda$ and $F$ is a face of $P$, then $F \in \Lambda$ and
\item[ii)] for any two polytopes $P$ and $Q$ in $\Lambda$, either $P \cap Q = \varnothing$ or $P \cap Q$ is a common face of both polytopes.
\end{itemize} 

The barycentric subdivision of a convex cell complex $Y$ is an ordered simplicial complex $bY$. An \emph{ordered simplicial complex} $X$ is a simplicial complex in which there is a partial order on the set of vertices $V(X)$ such that $X$ is the geometric realization of the simplicial complex $\Flag(V(X))$, the set of finite chains of $V(X)$ (see \cite[Appendix A]{Davis} for the geometric realization of the simplicial complex $\Flag(P)$ for a partially ordered set $P$). The partial order on $V(X)$ gives each edge in $X$ a direction. Two edges $a$ and $b$ are composable if $t(b) = i(a)$, where $t$ denotes the terminal point and $i$ denotes the initial point of an edge. 

Let $Y$ be a convex cell complex. A \emph{complex of spaces} over $Y$ is a connected CW complex $K$ and a map $p \colon K \longrightarrow Y$ such that 
\begin{itemize}
\item[i)] for each open cell $e_\alpha$ of $Y$, the pre-image $p^{-1}(e_\alpha)$ is a connected subcomplex of $K$ of the form $K_{e_\alpha}\times e_\alpha$, where $K_{e_\alpha}$ is the pre-image of the center of the cell $e_\alpha$,
\item[ii)] if $e_\beta$ is a cell contained in the boundary of $e_\alpha$, the induced map on fundamental groups $\pi_1(K_{e_\beta}) \longrightarrow \pi_1(K_{e_\alpha})$ is injective.
\end{itemize} 

If $K$ is a complex of spaces over $Y$, then there is an obvious bijection between the set of vertices of $bY$ and the set of spaces $K_{e_\alpha}$, for some cell $e_\alpha$ in $Y$. By the van Kampen theorem, the fundamental group $\pi_1(K)$ depends only on the fundamental groups of the spaces $K_{e_\alpha}$, the complex $Y$, and the maps $\pi_1(K_{e_\beta}) \longrightarrow \pi_1(K_{e_\alpha})$ whenever they are defined. This motivates the definition of complexes of groups and their fundamental groups.  

A \emph{complex of groups} $G(X)$ on an ordered simplicial complex $X$ is given by the following data.
\begin{enumerate}
\item[1)] For each $\sigma \in V(X)$, a group $G_\sigma$, called a \emph{local group}, is associated to $\sigma$.
\item[2)] For each $a \in E(X)$, an injective homomorphism $\psi \colon G_{i(a)} \longrightarrow G_{t(a)}$.
\item[3)] For each pair $a, b \in E(X)$ of composable edges, an element $g_{a,b} \in G_{t(a)}$ such that 
\begin{enumerate}
\item[a)] $\Ad(g_{a,b}) \psi_{ab} = \psi_a\psi_b$.
\item[b)] $\psi_{a}g_{b,c}g_{a,bc} = g_{a,b}g_{a,bc}$ for each triple of composable edges $a, b, c \in E(X)$. This is called the \emph{cocycle condition}.
\end{enumerate}
\end{enumerate}

If $K$ is complex of spaces over a complex $Y$, then associated to $K$ there is a complex of groups over $X := bY$. It is geometrically clear where the first two conditions in the definition of complexes of groups come from. Condition  ($3$a) corresponds to how the fiber corresponding to a $2$-dimensional cell is included into the graph of spaces corresponding to the $1$-skeleton of $X$. One can interpret $g_{a,b}$ as a change of base point, that is, the base point difference between two different ways of including $K_{i(b)}$ into $K_{t(a)}$: one is via $\psi_a\psi_b$, and the other is directly via $\psi_{ab}$. Condition ($3$b), or the cocyle condition, is the consistency condition that correspond to the change of basepoint between different inclusions of a higher dimensional cell group to a lower dimension cell group. 

The \emph{geometric realization} of a complex of groups $G(X)$ is a complex of spaces $K(X)$ over a complex $Y$ whose barycentric subdivision $bY$ is $X$ such that the associated complex of groups is $G(X)$. The \emph{fundamental group of a complex of groups} $G(X)$, denoted by $\pi_1(G(X))$ is the fundamental group of its geometric realization. The \emph{aspherical realization} or \emph{classifying space $BG(X)$ of a complex of groups} $G(X)$ is a geometric realization of $G(X)$ such that the fiber over each point is aspherical. For a given complex of group $G(X)$, the classifying space $BG(X)$ always exists and the homotopy type of its classifying spaces do not depend on the choice of classifying spaces \cite{Haefliger}. 

The nerve of a piecewise locally symmetric space $M$ is a convex cell complex $Y$. Let $X$ be the barycentric subdivision $bY$. The manifold $M$ has the structure of a complex of spaces since the inclusion of one stratum into another is injective on fundamental groups. Hence, $\pi_1(M)$ is the fundamental group of the associated complex of groups, say, $G(X)$. Then the vertex groups of $G(X)$ correspond to the fundamental groups of the strata of $M$. Since the fiber over each point of $Y$ in $M$ is an aspherical manifold, $M$ is an aspherical realization of the complex of groups $G(X)$, that is, $M$ is a $BG(X)$. 

The fundamental group of a complex of groups $G(X)$ has a presentation as follows. Let $T$ be a maximal tree in the $1$-skeleton of $X$. The generators are those in each local group $G_\sigma$, for $\sigma \in V(X)$, and all edges $a \in E(X)$, the set of edges of $X$. The relations are
\begin{itemize}
\item[a)] the relations in the group $G_\sigma$'s,
\item[b)] for each $a \in E(X)$ and $h \in G_{i(a)}$, then $\psi_a(h) = a^{-1}ha$,
\item[c)] for each pair of composable edges $a$ and $b$ with $c =ab$, then $c = bag_{a,b}$,
\item[d)] $a = 1$ for $a \in T$
\end{itemize}

Let $G(X)$ be a complex of groups over $X$. Like in Bass-Serre theory, where the universal cover of a graph of spaces is a graph of spaces over a simply connected graph (i.e. a tree, called the Bass-Serre tree associated to the given graph of groups), the universal cover $\widetilde{BG(X)}$ of $BG(X)$ is a simply connected complex of spaces over a simply connected complex $\widetilde{X}$ of the same dimension as $X$. In the same way as the fundamental group of a graph of groups acts without inversions on the Bass-Serre tree with quotient the original graph, the group $\pi_1(G(X))$ acts without inversions on $\widetilde{X}$ with quotient $X$. 

However, unlike in the theory of graphs of groups, the inclusion of $G_\sigma$ into the global group $\pi_1(G(X))$ needs not be injective. The stabilizer of a cell $\widetilde{\sigma}$ in $\widetilde{X}$ that projects to a cell $\sigma$ is not, in general, isomorphic to $G_\sigma$, but is isomorphic to the image of $G_\sigma$ under the canonical inclusion of $G_\sigma$ into $\pi_1(G(X))$. This gives rise to the notion of \emph{developability} for a complex of groups.

A complex of groups is called \emph{developable} if the inclusion of each vertex group $G_\sigma$ into the global group $\pi_1(G(X))$ is injective. This is equivalent to the usual notion of developable, as in \cite{Haefliger}, which is that there exists a complex $Y$ such that $X$ is the quotient of $Y$ under the action of a group without inversion. It follows that for a developable $G(X)$, the stabilizer in $\pi_1(G(X))$ of a cell $\widetilde{\sigma}$ in $\widetilde{X}$ that projects to a cell $\sigma$ is a conjugate of $G_\sigma$, which leads to the notion of local development of a vertex $\sigma \in V(X)$.  

For each $\sigma \in X$, the \emph{local development} of $\sigma$ is the following complex $\st(\widetilde{\sigma})$. Firstly we define a complex $\St(\widetilde{\sigma})$. The set of vertices of $\St(\widetilde{\sigma})$ is one-to-one correspondence with set $D_\sigma$ of cosets of $G_\delta$, for some subgroup $G_\delta$ of $G_\sigma$ and $G_\lambda$ that contains $G_\sigma$. Give $D_\sigma$ a partial ordering by inclusion. The complex $\St(\widetilde{\sigma})$ is the geometric realization of $\Flag(D_\sigma)$. Let $\widetilde{\sigma}$ be the vertex corresponding to $G_\sigma$.  We define $\st(\widetilde{\sigma})$ to be the star of  $\widetilde{\sigma}$ in $\St(\widetilde{\sigma})$, which is defined to be the union of the interior of the simplices in $X$ that meet $\widetilde{\sigma}$. The group $G_\sigma$ acts on $\st(\widetilde{\sigma})$ in the obvious way with quotient the star of $\st(\sigma)$ in $X$.

If $X$ is given a geodesic metric $d$, then this metric lifts to a geodesic metric $\widetilde{d}$ on $\widetilde{X}$. For example, for the case of a piecewise locally symmetric space $M$, the complex $X$ naturally admits a piecewise Euclidean metric. In general, suppose that $X$ is an $M_\kappa$-polyhedral complex, that is, each simplex in $X$ is metrized as a simplex in $M_\kappa$, the simply connected, complete, Riemannian manifold with constant curvature $\kappa$. The metric on each simplex of $X$ induces a metric on $\st(\widetilde{\sigma})$ for each $\sigma\in X$ that makes $\st(\widetilde{\sigma})$ an $M_\kappa$-polyhedral complex. If for each $\sigma \in V(X)$, the local development $\st(\widetilde{\sigma})$ has curvature $\leq \kappa$, then we say that $G(X)$ has curvature $\leq \kappa$. 

\begin{theorem}[\cite{Bridson}]
If a complex of group $G(X)$ is nonpositively curved, that is, it has curvature $\leq 0$, then it is developable.
\end{theorem}

If $G(X)$ is developable, then $BG(X)$ is homotopy equivalent to $\widetilde{X}$. To see this, consider the universal cover $\widetilde{BG(X)}$ of $BG(X)$. If the inclusions of local groups are injective, then for each $\sigma \in V(X)$, each connected component of pre-image of $K_\sigma$ in $\widetilde{BG(X)}$ is a copy of the universal cover $\widetilde{K}_\sigma$, which is contractible. So $\widetilde{BG(X)}$ is a complex of contractible spaces over $\widetilde{X}$. It follows that $\widetilde{BG(X)}$ is homotopy equivalent to the base complex $\widetilde{X}$ (see \cite[Lemma E.3.3]{Davis}. Therefore, $\widetilde{BG(X)}$ is contractible if and only if $\widetilde{X}$ is.

A sufficient condition for the complex $\widetilde{X}$ to be contractible is that it is $\CAT(0)$. For the case of piecewise locally symmetric spaces, the simplices in $\widetilde{X}$ are Euclidean convex polyhedron. So the associated complex of groups $G(X)$ is over a $M_0$-complex. Suppose that for each $\sigma \in X$, the local development of $\sigma$ has curvature $\leq 0$. Then $G(X)$ is nonpositively curved, and thus, is developable by the above theorem. This also implies that $\widetilde{X}$ has curvature $\leq 0$ since being nonpositively curved is a local condition. Moreover, $\widetilde{X}$ is simply connected, and thus is $\CAT(0)$ if it is nonpositively curved, and thus is contractible. It then follows that $BG(X)$ is contractible.

For a $M_0$-complex $X$, it is nonpositively curved if the link condition is satisfied. An $M_\kappa$-polyhedral complex is said to satisfied the \emph{link condition} if for each vertex $\sigma \in X$, the link complex $\Lk(\sigma, X)$ is a $\CAT(1)$ space. 

\begin{theorem}[\cite{Bridson}]
An $M_\kappa$-polyhedral complex $X$ with finite shape (i.e. there are finitely many isometry classes of simplices of $X$) has curvature $\leq \kappa$ if and only if it satisfies the link condition.
\end{theorem}

\bibliographystyle{amsplain}
\bibliography{bibliography}

\end{document}